\RequirePackage{fix-cm}
\documentclass[smallextended]{svjour3}       
\smartqed  
\usepackage{lipsum}
\usepackage{amsfonts}
\usepackage{amsmath}
\usepackage{amssymb}
\usepackage{graphicx}
\usepackage{xcolor}
\usepackage{epstopdf}
\usepackage{algorithmic}
\usepackage{dsfont}
\usepackage{mathtools}
\usepackage{multirow}

\usepackage{pdflscape}
\usepackage{tikz}
\newtheorem{remarkArt}{Remark}

\begin{document}

\title{Homogenization of a pseudo-parabolic system via a spatial-temporal decoupling: upscaling and corrector estimates for perforated domains
\thanks{We acknowledge the Netherlands Organisation of Scientific Research (NWO) for the MPE grant 657.000.004, the NWO Cluster Nonlinear Dynamics in Natural Systems (NDNS+) for funding a research stay of AJV at Karlstads Universitet and the Swedish Royal Academy of Science (KVA) for the Stiftelsen GS Magnusons fund grant MG2018-0020.}
}

\titlerunning{Pseudo-parabolic error estimates in perforated domain via space-time decoupling}        

\author{Arthur. J. Vromans         \and Fons van de Ven \and
        Adrian Muntean 
}


\institute{A.J. Vromans \and A.A.F. van de Ven\at
              Centre for Analysis, Scientific computing and Applications (CASA), Eindhoven University of Technology \\
              Den Dolech 2, 5612AZ, Eindhoven, The Netherlands\\
              \email{a.j.vromans@tue.nl;a.a.f.v.d.ven@tue.nl}           
           \and
           A.J. Vromans \and A. Muntean \at
              Department of Mathematics and Computer Science, Karlstad University\\
              Universitetsgatan 2, 651 88, Karlstad, Sweden\\
              \email{adrian.muntean@kau.se}
}

\date{Received: date / Accepted: date}

\maketitle

\begin{abstract} In this paper, we determine the convergence speed of an upscaling of a pseudo-parabolic system containing drift terms with scale separation of size $\epsilon\ll1$. Both the upscaling and convergence speed determination exploit a natural spatial-temporal decomposition, which splits the pseudo-parabolic system into a spatial elliptic partial differential equation and a temporal ordinary differential equation. We extend the applicability to space-time domains that are a product of spatial and temporal domains, such as a time-independent perforated spatial domain. Finally, for special cases we show convergence speeds for global times, i.e. $t\in\mathbf{R}_+$, by using time intervals that converge to $\mathbf{R}_+$ as $\epsilon\downarrow0$.
\end{abstract}
\section{Introduction}
\label{sec: intro}
Corrosion of concrete by acidic compounds is a problem for construction as corrosion can lead to erosion and degradation of the structural integrity of concrete structures \cite{Rendell2002}, \cite{Sand}. Structural failures and collapse as a result of concrete corrosion \cite{During1997}, \cite{GuFordMitchell}, \cite{Trethewey1995} is detrimental to society as it often impacts crucial infrastructure, typically leading to high costs \cite{Elsener}, \cite{Verdink}. Moreover, these failures can be avoided with sufficient monitoring and timely repairs based on \textit{a priori} calculations of the maximal lifespan of the concrete. These calculations have to take into account the heterogeneous nature of the concrete \cite{OrtizPopov1982}, the physical properties of the concrete \cite{Monteiro1996}, the corrosion reaction \cite{Taylor1997}, and the expansion/contraction behaviour of corroded concrete mixtures, see \cite{Bohm1998}, \cite{ClarelliFasanoNatalini}, \cite{FusiFarinaPrimicerio}. For example, the typical length scale of the concrete heterogeneities is much smaller than the typical length scale used in concrete construction \cite{OrtizPopov1982}. Moreover, concrete corrosion has a characteristic time that is also much smaller than the typical expected lifespan of concrete structures \cite{Taylor1997}. Hence, it is computationally expensive to use the heterogeneity length scale for simulations of concrete constructions such as bridges.   However, using averaging techniques in order to obtain effective properties on the typical length scale of concrete constructions, one can significantly decrease computational costs with the potential of not losing accuracy.\\
$\;$\\
Often a problem contains a hierarchy of separated scales: from a microscale via intermediate scales to a macroscale. With averaging techniques one can obtain effective behaviours at a higher scale from the underlying lower scale. For example, Ern and Giovangigli used averaging techniques on statistical distributions in kinetic chemical equilibrium regimes to obtain continuous macroscopic equations for mixtures, see \cite{ErnGiovangigli1998} or see Chapter 4 of \cite{Giovangigli1999} for a variety of effective macroscopic equations obtained with this averaging technique.\\
Of course, the use of averaging techniques to obtain effective macroscopic equations in mixture theory is by itself not new, see Fig 7.2 in \cite{Corwin} for an early application from 1934. The main problem with averaging techniques is choosing the right averaging technique for your problem. In this respect, homogenization can be regarded as a successful method, since it expresses conditions under which macroscale behaviour can be obtained from microscale behaviour and it has been successfully used to derive not only macroscale behaviour but also the convergence speed depending on the scale separation between the macroscale and the microscale.\\
$\;$\\
We perform homogenization via two-scale convergence as an averaging technique to obtain the macroscopic behaviour. Moreover, we use formal asymptotic expansions to determine the speed of convergence via so-called corrector estimates. These estimates follow a procedure similar to those used by Cioranescu and Saint Jean-Paulin in Chapter 2 of \cite{CioranescuStJeanPaulin1998}. Derivation via homogenization of constitutive laws, such as those arising from mixture theory, is a classical subject in homogenization, see \cite{Sanchez-Palencia}. Homogenization methods, upscaling, and corrector estimates are active research subjects due to the interdisciplinary nature of applying these mathematical techniques to real world problems and the complexities arising from the problem-specific constraints.\\ 
$\;$\\
The microscopic equations of our concrete corrosion model are conservation laws for mass and momentum for an incompressible mixture, see \cite{Vromans2018PAC} and \cite{VromansLIC} for details. The existence of weak solutions of this model was shown in \cite{Vromans2017CASA} 
and Chapter 2 of \cite{VromansLIC}. The parameter space dependence of the existence region for this model was explored in \cite{Vromans2018PAC}. The two-scale convergence for a subsystem of these microscopic equations, a pseudo-parabolic system, was shown in \cite{Vromans2018CASA}.
This paper handles the same pseudo-parabolic system as in \cite{Vromans2018CASA}
but on a perforated microscale domain.\\
$\;$\\
In \cite{Pesz-Showalter-Yi2009}, Peszy\'{n}ska, Showalter and Yi investigated the upscaling of a pseudo-parabolic system via two-scale convergence using a natural decomposition that splits the spatial and temporal behaviour. They looked at several different scale separation cases: classical case, highly heterogeneous case (also known as high-contrast case), vanishing time-delay case and Richards equation of porous media. These cases were chosen to showcase the ease with which upscaling could be done via this natural decomposition.\\
$\;$\\
In this paper, we point out that this natural decomposition of \cite{Pesz-Showalter-Yi2009} allows for the determination of the convergence speed via corrector estimates. Using such decomposition, the corrector estimates for the pseudo-parabolic equation follow straightforwardly from those of the spatially elliptic system with corrections due to the temporal first-order ordinary differential equation. The convergence speed we obtain, coincides for bounded spatial domains with known results for both elliptic systems and pseudo-parabolic systems on bounded temporal domains, see \cite{Reichelt2016}. Finally, we apply our results to a concrete corrosion model.\\
$\;$\\
The remainder of this paper is divided into seven parts:\\
\textbf{Section \ref{s: sec2}:} Notation and problem statement,\\
\textbf{Section \ref{sec: main}:} Main results,\\
\textbf{Section \ref{s: upscaling}:} Upscaling procedure,\\
\textbf{Section \ref{s: corrector}:} Corrector estimates,\\
\textbf{Section \ref{s: application}:} Application to a concrete corrosion model,\\
\textbf{Appendix \ref{app: constants}:} Exact forms of coefficients in corrector estimates,\\
\textbf{Appendix \ref{a: 2-scale}:} Introduction to two-scale convergence.\\
\section{Notation and problem statement}\label{s: sec2}
\subsection{Geometry of the medium and related function spaces}
We introduce the description of the geometry of the medium in question with a variant of the construction found in \cite{MunteanChalupecky2011}. Let $(0,T)$, with $T>0$, be a time-interval and $\Omega\subset\mathbf{R}^d$ for $d\in\{2,3\}$ be a simply connected bounded domain with a $C^2$-boundary $\partial\Omega$. Take $Y\subset\Omega$ a simply connected bounded domain, or more precisely there exists a diffeomorphism  $\gamma:\mathbf{R}^d\rightarrow\mathbf{R}^d$ such that $\text{Int}(\gamma([0,1]^d))=Y$.\\
We perforate $Y$ with a smooth open set $\mathcal{T} = \gamma(\mathcal{T}_0)$ for a smooth open set $\mathcal{T}_0\subset(0,1)^d$ such that $\overline{\mathcal{T}}\subset\overline{Y}$ with a $C^2$-boundary $\partial \mathcal{T}$ that does not intersect the boundary of $Y$, $\partial\mathcal{T}\cap\partial Y=\emptyset$, and introduce $Y^* = Y\backslash \overline{\mathcal{T}}$. Remark that $\partial \mathcal{T}$ is assumed to be $C^2$-regular.\\
Let $G_0$ be lattice\footnote{A lattice of a locally compact group $\mathbb{G}$ is a discrete subgroup $\mathbb{H}$ with the property that the quotient space $\mathbb{G}/\mathbb{H}$ has a finite invariant (under $\mathbb{G}$) measure. A discrete subgroup $\mathbb{H}$ of $\mathbb{G}$ is a group $\mathbb{H}\subsetneq\mathbb{G}$ under group operations of $\mathbb{G}$ such that there is (an open cover) a collection $\mathbb{C}$ of open sets $C\subsetneq\mathbb{G}$ satisfying $\mathbb{H}\subset \cup_{C\in\mathbb{C}}C$ and for all $C\in\mathbb{C}$ there is a unique element $h\in\mathbb{H}$ such that $h\in C$.} of the translation group $\mathcal{T}_d$ on $\mathbf{R}^d$ such that $[0,1]^d = \mathcal{T}_d/G_0$. Hence, we have the following properties: $\bigcup_{g\in G_0}g([0,1]^d)=\mathbf{R}^d$ and $(0,1)^d\cap g((0,1)^d)=\emptyset$ for all $g\in G_0$ not the identity-mapping. Moreover, we demand that the diffeomorphism $\gamma$ allows $G_\gamma:=\gamma\circ G_0\circ\gamma^{-1}$ to be a discrete subgroup of $\mathcal{T}_d$ with $\overline{Y} = \mathcal{T}_d/G_\gamma$.\\
Assume that there exists a sequence $(\epsilon_h)_h\subset(0,\epsilon_0)$ such that $\epsilon_h\rightarrow0$ as $h\rightarrow\infty$ (we omit the subscript $h$ when it is obvious from context that this sequence is mentioned). Moreover, we assume that for all $\epsilon_h\in(0,\epsilon_0)$ there is a set $G_\gamma^{\epsilon_h} = \{\epsilon_h g \text{ for }g\in G_\gamma\}$ with which we introduce $\mathcal{T}^{\epsilon_h} = \Omega\cap G_\gamma^{\epsilon_h}(\mathcal{T})$, the set of all holes and parts of holes inside $\Omega$. Hence, we can define the domain $\Omega^{\epsilon_h} = \Omega\backslash\mathcal{T}^{\epsilon_h}$ and we demand that $\Omega^{\epsilon_h}$ is connected for all $\epsilon_h\in(0,\epsilon_0)$. We introduce for all $\epsilon_h\in(0,\epsilon_0)$ the boundaries $\partial_{int}\Omega^{\epsilon_h}$ and $\partial_{ext}\Omega^{\epsilon_h}$ as $\partial_{int}\Omega^{\epsilon_h}= \bigcup_{g\in G_\gamma^{\epsilon_h}}\{\partial g(\overline{\mathcal{T}})\mid g(\overline{\mathcal{T}})\subset\Omega\}$ and  $\partial_{ext}\Omega^{\epsilon_h}=\partial\Omega^{\epsilon_h}\backslash \partial_{int}\Omega^{\epsilon_h}$. The first boundary contains all the boundaries of the holes fully contained in $\Omega$, while the second contains the remaining boundaries of the perforated region $\Omega$.\\
Note, $\mathcal{T}$ does not depend on $\epsilon$, since this could give rise to unwanted complicating effects such as treated in \cite{MarchenkoKrushlov}.\\
$\;$\\
Having the domains specified, we focus on defining the needed function spaces. We start by introducing $C_\#(Y)$, the space of continuous function defined on $Y$ and periodic with respect to $Y$ under $G_\gamma$. To be precise:
\begin{equation} \label{eq: C-periodic}
C_\#(Y) = \{f\in C(\mathbf{R}^d)|f\circ g = f\,\text{for all }g\in G_\gamma\}.
\end{equation}
Hence, the property ``$Y$-periodic'' means ``invariant under $G_\gamma$'' for functions defined on $Y$. Similarly the property ``$Y^*$-periodic'' means ``invariant under $G_\gamma$'' for functions defined on $Y^*$.\\
With $C_\#(Y)$ at hand, we construct Bochner spaces like $L^p(\Omega;C_\#(Y))$ for $p\geq1$ integer. For a detailed explanation of Bochner spaces, see Section 2.19 of \cite{KufnerFucik1977}. These types of Bochner spaces exhibit properties that hint at two-scale convergence, as is defined in Section \ref{s: two-scale}. Similar function spaces are constructed for $Y^*$ in an analogous way.

$\;$\\
Introduce the space
\begin{equation}\label{eq: V-space}
    \mathbb{V}_\epsilon = \{v\in H^1(\Omega^\epsilon)\mid v = 0\text{ on }\partial_{ext}\Omega^\epsilon \}
\end{equation}
equipped with the seminorm
\begin{equation}\label{eq: V-norm}
    \|v\|_{\mathbb{V}_\epsilon} = \|\nabla v\|_{L^2(\Omega^\epsilon)^d}.
\end{equation}
\begin{remarkArt}\label{r: poincare}
The seminorm in (\ref{eq: V-norm}) is equivalent to the usual $\mathcal{H}^1$-norm by the Poincar\'{e} inequality, see Lemma 2.1 on page 14 of \cite{CioranescuStJeanPaulin1998}. Moreover, this equivalence of norms is uniform in $\epsilon$.
\end{remarkArt}
For correct use of functions spaces over $Y$ and $Y^*$, we need an embedding result, which is based on an extension operator. The following theorem and corollary are Theorem 2.10 and Corollary 2.11 in Chapter 2 of \cite{CioranescuStJeanPaulin1998}.
\begin{theorem}\label{t: extension} Suppose that the domain $\Omega^\epsilon$ is such that $\mathcal{T}\subset Y$ is a smooth open set with a $C^2$-boundary that does not intersect the boundary of $Y$ and such that the boundary of $\mathcal{T}^\epsilon$ does not intersect the boundary of $\Omega$. Then there exists an extension operator $\mathcal{P}^\epsilon$ and a constant $C$ independent of $\epsilon$ such that
\begin{equation}\label{eq: extendregul}
\mathcal{P}^\epsilon\in\mathcal{L}(L^2(\Omega^\epsilon);L^2(\Omega))\cap\mathcal{L}(\mathbb{V}_\epsilon;H^1_0(\Omega)),
\end{equation}
and for any $v\in\mathbb{V}_\epsilon$, we have the bounds
\begin{equation}
    \|\mathcal{P}^\epsilon v\|_{L^2(\Omega)}\leq C\|v\|_{L^2(\Omega_\epsilon)},\quad\|\nabla\mathcal{P}^\epsilon v\|_{L^2(\Omega)^d}\leq C\|\nabla v\|_{L^2(\Omega_\epsilon)^d}.
\end{equation}
\end{theorem}
\begin{corollary}\label{c: extension} There exists a constant $C$ independent of $\epsilon$ such that for all $v\in\mathbb{V}_\epsilon$
\begin{equation}
    \|\mathcal{P}^\epsilon v\|_{H^1_0(\Omega)}\leq C\|v\|_{\mathbb{V}_\epsilon}.
\end{equation}
\end{corollary}
Introduce the notation $\hat{\cdot}$, a hat symbol, to denote extension via the extension operator $\mathcal{P}^\epsilon$.
\subsection{The Neumann problem (\ref{eq: Aeps-sys})-(\ref{eq: Vhole=0})}
The notation $\nabla = (\frac{\mathrm{d}}{\mathrm{d}x_1},\ldots,\frac{\mathrm{d}}{\mathrm{d}x_d})$ denotes the vectorial total derivative with respect to the components of $\vec{x}  = (x_1,\ldots,x_d)^\top$ for functions depending on both $\vec{x}$ and $\vec{x}/\epsilon$. Spatial vectors have $d$ components, while variable vectors have $N$ components. Tensors have $d^iN^j$ components for $i$, $j$ nonnegative integers. Furthermore, the notation
\begin{equation}\label{eq: notation}c^\epsilon(t,\vec{x}) = c(t,\vec{x},\vec{x}/\epsilon)\end{equation} is used for the $\epsilon$-independent functions $c(t,\vec{x},\vec{y})$ in assumption (A1) further on. Moreover, the spatial inner product is denoted with $\cdot$, while the variable inner product is just seen as a product or operator acting on a variable vector or tensor.\\
$\;$\\
Let $T>0$. We consider the following Neumann problem posed on $(0,T)\times\Omega^\epsilon:$
\begin{subequations}
\begin{align}
    (\mathcal{A}^\epsilon \vec{V}^\epsilon)_\alpha&:=\sum_{\beta=1}^NM_{\alpha\beta}^\epsilon V^\epsilon_\beta-\sum_{i,j=1}^d\frac{\mathrm{d}}{\mathrm{d}x_i}\left(E_{ij}^\epsilon\frac{\mathrm{d} V^\epsilon_\alpha}{\mathrm{d} x_j}+\sum_{\beta=1}^N\tens{D}^\epsilon_{i\alpha\beta}V^\epsilon_\beta\right)\cr
    &= \mathcal{H}^\epsilon_\alpha+\sum_{\beta=1}^N\left(K^\epsilon_{\alpha\beta}U^\epsilon_\beta+\sum_{i=1}^d\tilde{J}^\epsilon_{i\alpha\beta}\frac{\mathrm{d} U^\epsilon_\beta}{\mathrm{d} x_i}\right)=:(\mathcal{H}^\epsilon \vec{U}^\epsilon)_\alpha,\label{eq: Aeps-sys}\\
    (\mathcal{L} \vec{U}^\epsilon)_\alpha&:=\frac{\partial U^\epsilon_\alpha}{\partial t} + \sum_{\beta=1}^NL_{\alpha\beta}U^\epsilon_\beta = \sum_{\beta=1}^NG_{\alpha\beta}V^\epsilon_\beta,\label{eq: L-sys}
\end{align}
\end{subequations}
with the boundary conditions
\begin{subequations}
\begin{align}
    U_\alpha^\epsilon&=U_\alpha^*&\text{ in }\{0\}\times\Omega^\epsilon,\qquad\!\!\quad\label{eq: init-eps}\\
    V_\alpha^\epsilon& = 0&\text{ on }(0,T)\times\partial_{ext}\Omega^\epsilon,\label{eq: V=0}\\
    \frac{\mathrm{d} V^\epsilon_\alpha}{\mathrm{d}\nu_{\tens{D}^\epsilon}}&:=\sum_{i=1}^d\left(\sum_{j=1}^dE_{ij}^\epsilon\frac{\mathrm{d} V^\epsilon_\alpha}{\mathrm{d} x_j}+\sum_{\beta=1}^ND_{i\alpha\beta}^\epsilon V^\epsilon_\beta\right)n_i^\epsilon = 0&\text{ on }(0,T)\times\partial_{int}\Omega^\epsilon,\label{eq: Vhole=0}
\end{align}
\end{subequations}
for $\alpha\in\{1,\ldots,N\}$ or, in short-hand notation, this reads:
\begin{equation}\label{eq: neumann-operators}
\begin{dcases}
    \mathcal{A}^\epsilon \vec{V}^\epsilon\!:=\tens{M}^\epsilon\vec{V}^\epsilon-\nabla\cdot\left(\tens{E}^\epsilon\cdot\nabla\vec{V}^\epsilon+\tens{D}^\epsilon\vec{V}^\epsilon\right)\cr
    \qquad\quad\!= \vec{H}^\epsilon+\tens{K}^\epsilon\vec{U}^\epsilon+\tilde{\tens{J}}^\epsilon\cdot\nabla\vec{U}^\epsilon=:\mathcal{H}^\epsilon \vec{U}^\epsilon&\text{ in }\,(0,T)\times\Omega^\epsilon,\!\!\qquad\cr
    \mathcal{L} \vec{U}^\epsilon\!\!\!\!\quad:=\frac{\partial \vec{U}^\epsilon}{\partial t} + \tens{L}\vec{U}^\epsilon = \tens{G}\vec{V}^\epsilon&\text{ in }\,(0,T)\times\Omega^\epsilon,\!\!\qquad\cr
    \quad\vec{U}^\epsilon\,\,=\vec{U}^*&\text{ in }\,\quad\{0\}\times\Omega^\epsilon,\!\!\qquad\cr
    \quad\vec{V}^\epsilon \,\,= 0&\text{ on }(0,T)\times\partial_{ext}\Omega^\epsilon,\cr
    \frac{\mathrm{d} \vec{V}^\epsilon}{\mathrm{d}\nu_{\tens{D}^\epsilon}}\!\!\!\!\quad= \left(\tens{E}^\epsilon\cdot\nabla\vec{V}^\epsilon+\tens{D}^\epsilon\vec{V}^\epsilon\right)\cdot\vec{n}^\epsilon=0&\text{ on }(0,T)\times\partial_{int}\Omega^\epsilon.
\end{dcases}
\end{equation}
\subsection{Assumptions}
Consider the following technical requirements for the coefficients arising in the Neumann problem (\ref{eq: Aeps-sys}) - (\ref{eq: Vhole=0}).
\begin{itemize}
    \item[(A1)] For all $\alpha,\beta\in\{1,\ldots,N\}$ and for all $i,j\in\{1,\ldots,d\}$, we assume:
    \begin{equation}\label{eq: regularity1}
    \begin{array}{rcl}
     M_{\alpha\beta},H_\alpha,K_{\alpha\beta},J_{i\alpha\beta}&\in& L^\infty(\mathbf{R}_+;W^{2,\infty}(\Omega;C^2_\#(Y^*))),\cr
     E_{ij},D_{i\alpha\beta}&\in& L^\infty(\mathbf{R}_+;W^{3,\infty}(\Omega;C^3_\#(Y^*))),\cr
     L_{\alpha\beta},G_{\alpha\beta}&\in& L^\infty(\mathbf{R}_+;W^{4,\infty}(\Omega)),\cr
     U^*_\alpha&\in& W^{4,\infty}(\Omega),
    \end{array}
    \end{equation}
    with $\tilde{\tens{J}}^\epsilon = \epsilon \tens{J}^\epsilon$; see Remark \ref{r: A1} further on.
    \item[(A2)] The tensors $\tens{M}$ and $\tens{E}$ have a linear sum decomposition\footnote{For real symmetric matrices $\tens{M}$ and $\tens{E}$, the finite dimensional version of the spectral theorem states that they are diagonalizable by orthogonal matrices. Since $\tens{M}$ acts on the variable space $\mathbf{R}^N$, while $\tens{E}$ acts on the spatial space $\mathbf{R}^d$, one can simultaneously diagonalize both real symmetric matrices. For general real matrices $\tens{M}$ and $\tens{E}$ the linear sum decomposition in symmetric and skew-symmetric matrices allows for a diagonalization of the symmetric part. The orthogonal matrix transformations necessary to diagonalize the symmetric part does not modify the regularity of the domain $\Omega$, of the perforated periodic cell $Y^*$ or of the coefficients of $\tens{D}$, $\vec{H}$, $\tens{K}$, $\tens{J}$, $\tens{L}$, or $\tens{G}$. Hence, we are allowed to assume a linear sum decomposition of $\tens{M}$ and $\tens{E}$ in a diagonal and a skew-symmetric matrix.} with a skew-symmetric matrix and a diagonal matrix with the diagonal elements of $\tens{M}$ and $\tens{E}$ denoted by $M_\alpha, E_i\in L^\infty(\mathbf{R}_+\times\Omega;C_\#(Y^*))$, respectively, satisfying $M_\alpha>0$, $E_i>0$ and $1/M_\alpha,1/E_i\in L^\infty(\mathbf{R}_+\times\Omega\times Y^*)$.
    \item[(A3)] The inequality
    \begin{equation}
        \|D_{i\beta\alpha}^\epsilon\|_{L^\infty(\mathbf{R}_+\times\Omega^\epsilon;C_\#(Y^*))}^2<\frac{4m_\alpha e_i}{dN^2}\label{H: eq:  ineq}
    \end{equation}
    holds with
    \begin{equation}\label{eq: A3}
        \frac{1}{m_\alpha} = \left\|\frac{1}{M_\alpha}\right\|_{L^\infty(\mathbf{R}_+\times\Omega\times Y^*)}\quad\text{ and }\quad\frac{1}{e_i} = \left\|\frac{1}{E_i}\right\|_{L^\infty(\mathbf{R}_+\times\Omega\times Y^*)}
    \end{equation}
    for all $\alpha,\beta\in\{1,\ldots,N\}$, for all $i\in\{1,\ldots,d\}$, and for all $\epsilon\in(0,\epsilon_0)$.
    \item[(A4)] The perforation holes do not intersect the boundary of $\Omega$: \begin{equation*}\partial\mathcal{T}^\epsilon\cap\partial\Omega = \emptyset\text{ for a given sequence }\epsilon\in(0,\epsilon_0).
    \end{equation*}
    \end{itemize}
\begin{remarkArt} \label{r: A1}
    The dependence $\tens{J}^\epsilon=\epsilon\tens{J}^\epsilon$ was chosen to simplify both existence and uniqueness results and arguments for bounding certain terms. The case $\tens{J}^\epsilon = \tens{J}^\epsilon$ can be treated with the proofs outlined in this paper if additional cell functions are introduced and special inequalities similar to the Poincar\'{e}-Wirtinger inequality are used. See (\ref{eq: cellV1}) onward in Section \ref{s: upscaling} for the introduction of cell functions.
\end{remarkArt}
\begin{remarkArt}\label{r: A3}
    Satisfying inequality (\ref{H: eq:  ineq}) implies that the same inequality is satisfied for the $Y^*$-averaged functions $\overline{D_{i\beta\alpha}^\epsilon}$, $\overline{M_{\beta\alpha}^\epsilon}$, and $\overline{\tens{E}^\epsilon_{ij}}$ in $L^\infty(\mathbf{R}_+\times\Omega)$, where we used the following notion of $Y^*$-averaged functions
\begin{equation}\label{eq: average}
    \overline{f}(t,\vec{x}) = \frac{1}{|Y|}\int_{Y^*}f(t,\vec{x},\vec{y})\mathrm{d}\vec{y}.
\end{equation}
\end{remarkArt}
\begin{remarkArt}\label{r: A4}
Assumption (A4) implies the following identities for the given sequence $\epsilon\in(0,\epsilon_0)$:
\begin{equation}\label{eq: domainbounds}
\partial_{int}\Omega^\epsilon = \partial\mathcal{T}^\epsilon\cap\Omega,\quad \partial_{ext}\Omega^\epsilon = \partial\Omega.
\end{equation}
Without (A4) perforations would intersect $\partial\Omega$. One must then decide which parts of the boundary of the intersected cell $Y^*$ satisfies which boundary condition: (\ref{eq: V=0}) or (\ref{eq: Vhole=0}). This leads to non-trivial situations, that ultimately affects the corrector estimates in non-trivial ways.
\end{remarkArt}
\begin{theorem}\label{t: unexNeu}Under assumptions (A1)-(A4), there exist a solution pair\\ $(\vec{U}^\epsilon,\vec{V}^\epsilon) \in H^1((0,T)\times\Omega^\epsilon)^N\times L^\infty((0,T);\mathbb{V}_\epsilon\cap H^2(\Omega^\epsilon))^N$ satisfying the\\ Neumann problem (\ref{eq: Aeps-sys})-(\ref{eq: Vhole=0}).
\end{theorem}
\begin{proof}
For $\tens{K}^\epsilon = \tens{M}^\epsilon\tens{G}^{-1}\tens{L}$, $\tens{J}^\epsilon=\tens{0}$ and $d=1$ the result follows by Theorem 1 in \cite{Vromans2017CASA}.
For non-perforated domains the result follows by either Theorem 1 in \cite{Vromans2018CASA}
or Theorem 7 in Chapter 4 of \cite{VromansLIC}.\\
For perforated domains, the result follows similarly. An outline of the proof is as follows. First, time-discretization is applied such that $\mathcal{A}^\epsilon\vec{V}^\epsilon$ at $t=k\tens{Z} t$ equals $\mathcal{H}^\epsilon \vec{U}^\epsilon$ at $t= (k-1)\tens{Z} t$ and $\mathcal{L}\vec{U}^\epsilon$ at $t=k\tens{Z} t$ equals $\tens{G}\vec{V}^\epsilon$ at $t= (k-1)\tens{Z} t$. This is an application of the Rothe method. Under assumptions (A1)-(A4), testing $\mathcal{A}^\epsilon\vec{V}^\epsilon$ with a function $\phi$ yields a continuous and coercive bilinear form on $H^1(\Omega^\epsilon)^N$, while testing $\mathcal{L}\vec{U}^\epsilon$ with a function $\psi$ yields a continuous and coercive bilinear form on $L^2(\Omega^\epsilon)^N$. Hence, Lax-Milgram leads to the existence of a solution at each time slice $t = k\tens{Z} t$.\\
Choosing the right functions for $\phi$ and $\psi$ and using a discrete version of Gronwall's inequality we obtain upper bounds of $\vec{U}^\epsilon$ and $\vec{V}^\epsilon$ independent of $\tens{Z} t$. Linearly interpolating the time slices, we find that the $\tens{Z} t$-independent time slices guarantee the existence of continuous weak limits. Due to sufficient regularity, we even obtain strong convergence and existence of boundary traces. Then the continuous weak limits are actually weak solutions of our Neumann problem (\ref{eq: Aeps-sys})-(\ref{eq: Vhole=0}). The uniqueness follows by the linearity of our Neumann problem (\ref{eq: Aeps-sys})-(\ref{eq: Vhole=0}).
\qed\end{proof}
\section{Main results}
\label{sec: main}
Two special length scales are involved in the Neumann problem (\ref{eq: Aeps-sys})-(\ref{eq: Vhole=0}): The variable $\vec{x}$ is the ``macroscopic'' scale, while $\vec{x}/\epsilon$ represents the ``microscopic'' scale. This leads to a double dependence of parameter functions (and, hence, of the solutions to the model equations), on both the macroscale and the microscale. For example, if $\vec{x}\in\Omega^\epsilon$, by the definition of $\Omega^\epsilon$, there exists $g\in G_\gamma$ such that $\vec{x}/\epsilon = g(\vec{y})$ with $\vec{y}\in Y^*$. This suggests that we look for a formal asymptotic expansion of the form
\begin{subequations}
\begin{align}
    \vec{V}^\epsilon(t,\vec{x}) &=& \vec{V}^0\left(t,\vec{x},\frac{\vec{x}}{\epsilon}\right)+\epsilon \vec{V}^1\left(t,\vec{x},\frac{\vec{x}}{\epsilon}\right)+\epsilon^2 \vec{V}^2\left(t,\vec{x},\frac{\vec{x}}{\epsilon}\right)+\cdots,\label{eq: V-expansion1}\\
    \vec{U}^\epsilon(t,\vec{x}) &=& \vec{U}^0\left(t,\vec{x},\frac{\vec{x}}{\epsilon}\right)+\epsilon \vec{U}^1\left(t,\vec{x},\frac{\vec{x}}{\epsilon}\right)+\epsilon^2 \vec{U}^2\left(t,\vec{x},\frac{\vec{x}}{\epsilon}\right)+\cdots\label{eq: U-expansion1}
    \end{align}
\end{subequations}
with $\vec{V}^j(t,\vec{x},\vec{y})$, $\vec{U}^j(t,\vec{x},\vec{y})$ defined for $t\in\mathbf{R}_+$, $\vec{x}\in\Omega^\epsilon$ and $\vec{y}\in Y^*$ and $Y^*$-periodic (i.e. $\vec{V}^j$, $\vec{U}^j$ are periodic with respect to $G_\gamma^{\epsilon}$).
\begin{theorem}\label{t: two-scale} Let assumptions (A1)-(A4) hold. For all $T\in\mathbf{R}_+$ there exist a unique pair $(\vec{U}^\epsilon,\vec{V}^\epsilon)\in H^1((0,T)\times\Omega^\epsilon)^N\times L^\infty((0,T);\mathbb{V}_\epsilon)^N$ satisfying the Neumann problem (\ref{eq: Aeps-sys})-(\ref{eq: Vhole=0}).
Moreover, for $\epsilon\downarrow0$
\begin{subequations}
\begin{align}
    \hat{\vec{U}}^\epsilon\overset{2}{\longrightarrow}\vec{U}^0&\text{ in }H^1((0,T)\times\Omega)^N\text{ and}\label{eq: uconverge1}\\
    \hat{\vec{V}}^\epsilon\overset{2}{\longrightarrow}\vec{V}^0&\text{ in }L^\infty((0,T);H^1_0(\Omega))^N.\label{eq: vconverge1}
\end{align}
\end{subequations}
This implies
\begin{subequations}
\begin{align}
    \hat{\vec{U}}^\epsilon\rightharpoonup \vec{U}^0&\text{ in }H^1((0,T)\times\Omega)^N\text{ and}\label{eq: uconverge2}\\
    \hat{\vec{V}}^\epsilon\rightharpoonup \vec{V}^0&\text{ in }L^\infty((0,T);H^1_0(\Omega))^N\label{eq: vconverge2}
\end{align}
\end{subequations}
for $\epsilon\downarrow0$.\end{theorem}
\begin{proof}
See Section \ref{s: upscaling} for the full details and \cite{Vromans2018CASA}
for a short proof of the two-scale convergence for a non-perforated setting.
\qed\end{proof}
Additionally, we are interested in deriving the speed of convergence of the formal asymptotic expansion. Boundary effects are expected to occur due to intersection of the external boundary with the perforated periodic cells. Hence, a cut-off function is introduced to remove this part from the analysis.\\
Let $M_\epsilon$ be the cut-off function defined by
\begin{equation}\label{eq: cutoff}
    \begin{dcases}
    M_\epsilon\in\mathcal{D}(\Omega),\cr
    M_\epsilon = 0 &\text{ if }\mathrm{dist}(\vec{x},\partial\Omega)\leq\,\,\,\epsilon\,\mathrm{diam}(Y),\cr
    M_\epsilon = 1 &\text{ if }\mathrm{dist}(\vec{x},\partial\Omega)\geq2\epsilon\,\mathrm{diam}(Y),\cr
    \epsilon\left|\frac{\mathrm{d} M_\epsilon}{\mathrm{d} x_i}\right|\leq C&i\in\{1,\ldots,d\}.
    \end{dcases}
\end{equation}
We refer to
\begin{subequations}
\begin{align}
\vec{\Phi}^\epsilon&=\vec{V}^\epsilon-\vec{V}^0-M_\epsilon(\epsilon \vec{V}^1+\epsilon^2 \vec{V}^2),\label{eq: t: PHI}\\
\vec{\Psi}^\epsilon&=\vec{U}^\epsilon-\vec{U}^0-M_\epsilon(\epsilon \vec{U}^1+\epsilon^2 \vec{U}^2)\label{eq: t: PSI}
\end{align}
\end{subequations}
as error functions. Now, we are able to state our convergence speed result.
\begin{theorem}\label{t: corrector}Let assumptions (A1)-(A4) hold. There exist constants $l\geq0$, $\kappa\geq0$, $\tilde{\kappa}\geq0$, $\lambda\geq0$ and $\mu\geq0$ such that
\begin{subequations}
    \begin{align}
        \|\vec{\Phi}^\epsilon\|_{\mathbb{V}_\epsilon^N}(t)&\!\leq\!\mathcal{C}(\epsilon,t) ,\!\label{eq: t: corrector1}\\
        \|\vec{\Psi}^\epsilon\|_{H^1(\Omega^\epsilon)^N}(t)&\!\leq\!\mathcal{C}(\epsilon,t)\sqrt{t_le^{lt}}\label{eq: t: corrector2}
    \end{align}
\end{subequations}
with
\begin{equation}
    \mathcal{C}(\epsilon,t) = C(\epsilon^{\frac{1}{2}}\!+\!\epsilon^{\frac{3}{2}})\left[1\!+\!\epsilon^{\frac{1}{2}}(1\!+\!\tilde{\kappa}e^{\lambda t})(1\!+\!\kappa(1\!+\! t_le^{lt}))\right]\!\exp\!\left(\mu t_le^{lt}\right)\label{eq: bounddef}
\end{equation}
where $C$ is a constant independent of $\epsilon$ and $t$, and $t_l = \min\{1/l,t\}$.
\end{theorem}
\begin{remarkArt}\label{r: finiteT}
The upper bounds in (\ref{eq: t: corrector1}) and (\ref{eq: t: corrector2}) are $\mathcal{O}(\epsilon^{\frac{1}{2}})$ for $\epsilon$-independent finite time intervals. We call this type of bounds corrector estimates.
\end{remarkArt}
The corrector estimate of $\vec{\Phi}^\epsilon$ in Theorem \ref{t: corrector} becomes that of the classic linear elliptic system for $\tens{K}=\tens{0}$ and $\tens{J}=0$. This is because $\tens{K}=\tens{0}$ and $\tens{J}=0$ imply $\tilde{\kappa}=\kappa = \mu = 0$, see Appendix \ref{app: constants}. See \cite{CioranescuStJeanPaulin1998} for the classical approach to corrector estimates of elliptic systems in perforated domains and \cite{Meshkova-Suslina2016} for a spectral approach in non-perforated domains.
\begin{corollary}\label{c: corrector}
Under the assumptions of Theorem \ref{t: corrector},
\begin{subequations}
    \begin{align}
        \|\hat{\vec{V}}^\epsilon\!-\!\vec{V}^0\|_{H^1_0(\Omega)^N}(t)&\!\leq\!\mathcal{C}(\epsilon,t),\label{eq: c: corrector1}\\
        \!\!\!\!\!\!\!\!\!\|\hat{\vec{U}}^\epsilon\!-\!\vec{U}^0\|_{H^1(\Omega)^N}(t)&\!\leq\!\mathcal{C}(\epsilon,t)\sqrt{t_le^{lt}}\label{eq: c: corrector2}
    \end{align}
\end{subequations}
hold, where $C$ is a constant independent of $\epsilon$ and $t$.
\end{corollary}
According to Remark \ref{r: finiteT}, $\epsilon$-independent finite time intervals yield $\mathcal{O}(\epsilon^{\frac{1}{2}})$ corrector estimates. Is it, then, possible to have a converging corrector estimate for diverging time intervals in the limit $\epsilon\downarrow0$? The next theorem answers this question positively.\\
\begin{theorem}\label{t: timeinterval}
If $l>0$, we introduce the rescaled time $\tau\ln\left(\frac{1}{\epsilon}\right)=\exp(lt)\geq1$ and $q\in(0,\frac{1}{2})$ independent of both $\epsilon$ and $t$ satisfying $0<\mu\tau/l<\frac{1}{2}-q$. Then, for $0<\epsilon< \exp(-\frac{2\mu}{(1-2q)l})$, we have the corrector bounds
\begin{subequations}
    \begin{align}
        \|\vec{\Phi}^\epsilon\|_{\mathbb{V}_\epsilon^N}(t)&=\mathcal{O}\left(\epsilon^{\frac{1}{2}-\frac{\mu}{l}\tau}\right)=o(1)=\omega\left(\epsilon^{\frac{1}{2}}\right),\label{eq: theocorrV1}\\
        \|\vec{\Psi}^\epsilon\|_{H^1(\Omega^\epsilon)^N}(t)&=\mathcal{O}\left(\epsilon^{\frac{1}{2}-\frac{\mu}{l}\tau}\right)\mathcal{O}\left(\frac{\epsilon^{-q}}{q}\right)=o(1)=\omega\left(\epsilon^{\frac{1}{2}}\right)\label{eq: theocorrU1}
    \end{align}
\end{subequations}
as $\epsilon\downarrow0$.\\
If $l=0$, we introduce the rescaled time $\tau\ln\left(\frac{1}{\epsilon}\right)=t\geq0$ and $p,q\in(0,\frac{1}{2})$ independent of both $\epsilon$ and $t$ satisfying $0<\max\{\mu\tau,(\lambda+\mu)\tau+p-\frac{1}{2}\}<\frac{1}{2}-q$. Then, for $0<\epsilon<1$, we have the corrector bounds
\begin{subequations}
    \begin{align}
        \|\vec{\Phi}^\epsilon\|_{\mathbb{V}_\epsilon^N}(t)&=\mathcal{O}\left(\epsilon^{\frac{1}{2}-\mu\tau}\right)+\mathcal{O}\left(\epsilon^{1-(\lambda+\mu)\tau}\right)\mathcal{O}\left(\frac{\epsilon^{-p}}{p}\right),\label{eq: theocorrV2}\\
        \|\vec{\Psi}^\epsilon\|_{H^1(\Omega^\epsilon)^N}(t)&=\left[\mathcal{O}\left(\epsilon^{\frac{1}{2}-\mu\tau}\right)+\mathcal{O}\left(\epsilon^{1-(\lambda+\mu)\tau}\right)\mathcal{O}\left(\frac{\epsilon^{-p}}{p}\right)\right]\mathcal{O}\left(\frac{\epsilon^{-q}}{q}\right)\label{eq: theocorrU2}
    \end{align}
\end{subequations}
as $\epsilon\downarrow0$. If, additionally, $\kappa=0$ holds, then the bounds change to
\begin{subequations}
    \begin{align}
        \|\vec{\Phi}^\epsilon\|_{\mathbb{V}_\epsilon^N}(t)&=\mathcal{O}\left(\epsilon^{\min\{\frac{1}{2},1-\lambda\tau\}}\right),\label{eq: theocorrV3}\\
        \|\vec{\Psi}^\epsilon\|_{H^1(\Omega^\epsilon)^N}(t)&=\mathcal{O}\left(\epsilon^{\min\{\frac{1}{2},1-\lambda\tau\}}\right)\mathcal{O}\left(\frac{\epsilon^{-q}}{q}\right).\label{eq: theocorrU3}
    \end{align}
\end{subequations}
\end{theorem}
\begin{proof} Insert the definition of the rescaled time into (\ref{eq: t: corrector1}) and (\ref{eq: t: corrector2}), use $t_l = \min\{1/l,t\} = t$ for $l=0$ and $t_l\leq 1/l$ for $l>0$. Now one obtains the product $\epsilon^\tens{Z}\ln(1/\epsilon)$ at several locations, which has a single maximal value of $\frac{1}{\tens{Z} e}$ at $\ln\left(\frac{1}{\epsilon}\right) = \frac{1}{\tens{Z}}$. The minimum function is needed since $\mathcal{O}(\epsilon^{r})+\mathcal{O}(\epsilon^{s}) =\mathcal{O}(\epsilon^{\min\{r,s\}})$.
The small $o$ and small $\omega$ orders are upper and lower asymptotic convergence speeds, respectively, for $\epsilon\downarrow0$. The upper bound for $\epsilon$ is needed to guarantee that the interval for $\tau$ corresponds to $t\geq0$.
\qed
\end{proof}
Theorem \ref{t: timeinterval} indicates that convergence can be retained for certain diverging sequences of time-intervals. Consequently, appropriate rescalings of the time variable yield upscaled systems and convergence rates for systems with regularity conditions different from those in assumptions (A1) - (A3).\\
\begin{remarkArt}\label{r: rescale}
The tensors $\tens{L}$ and $\tens{G}$ are not dependent on $\epsilon$ nor are unbounded functions of $t$. If such a dependence or unbounded behaviour does exist, then bounds similar to those stated in Theorem \ref{t: corrector} are still valid in a new time-variable $s\in I\subset\mathbf{R}_+$ if an invertible $C^1$-map $f_\epsilon$ from $t\in\mathbf{R}_+$ to $s$ exists such that tensors
$(\tens{L}^\epsilon/f'_\epsilon)\circ f_\epsilon^{-1}$, $(\tens{G}^\epsilon/f'_\epsilon)\circ f_\epsilon^{-1}$, $\tens{M}^\epsilon\circ f_\epsilon^{-1}$, $\tens{E}^\epsilon\circ f_\epsilon^{-1}$, $\tens{D}^\epsilon\circ f_\epsilon^{-1}$, $\tens{H}^\epsilon\circ f_\epsilon^{-1}$, $\tens{K}^\epsilon\circ f_\epsilon^{-1}$, and $\tens{J}^\epsilon\circ f_\epsilon^{-1}$
satisfy (A1)-(A3).\\
Moreover, if $f_\epsilon(\mathbf{R}_+) = \mathbf{R}_+$ for $\epsilon>0$ small enough, then the bounds of Theorem \ref{t: timeinterval} are valid as well with $\tau$ defined in terms of $s$.
\end{remarkArt}

\section{Upscaling procedure}
\label{s: upscaling}
Upscaling of the Neumann problem (\ref{eq: Aeps-sys})-(\ref{eq: Vhole=0}) can be done by many methods, e.g. via asymptotic expansions or two-scale convergence in suitable function spaces. We proceed in four steps:
\begin{itemize}
    \item[1.] \textbf{Existence and uniqueness of $(\vec{U}^\epsilon,\vec{V}^\epsilon)$.}\\
    We rely on Theorem \ref{t: unexNeu}.
    \item[2.] \textbf{Obtain $\epsilon$-independent bounds for $(\vec{U}^\epsilon,\vec{V}^\epsilon)$.}\\
    See Section \ref{H: sec: 4}.
    \begin{itemize}
    \item[a.] Obtain \textit{a priori} estimates for $(\vec{U}^\epsilon,\vec{V}^\epsilon)$. See Lemma \ref{H: l: : apriori}.
    \item[b.] Obtain $\epsilon$-independent bounds for $(\vec{U}^\epsilon,\vec{V}^\epsilon)$. See Theorem \ref{H: t: existuniq}.
    \end{itemize}
    \item[3.] \textbf{Upscaling via two-scale convergence.}\\
    See Section \ref{H: sec: 5}.
    \begin{itemize}
    \item[a.] Two-scale limit of $(\vec{U}^\epsilon,\vec{V}^\epsilon)$ for $\epsilon\downarrow0$. See Lemma \ref{H: l: : two-scale-conv}.
    \item[b.] Two-scale limit of problem (\ref{eq: Aeps-sys})-(\ref{eq: Vhole=0}) for $\epsilon\!\,\!\,\!\downarrow\!\,\!\,\!0$. See Theorem \ref{H: t: upscale}.
    \end{itemize}
    \item[4.] \textbf{Upscaling via asymptotic expansions and relating to two-scale convergence.}\\ See Section \ref{sec: 4.3}.
    \begin{itemize}
    \item[a.] Expand (\ref{eq: Aeps-sys}) and $(\vec{U}^\epsilon,\vec{V}^\epsilon)$. See equations (\ref{eq: neumann-operator})-(\ref{eq: L1sys}).
    \item[b.] Obtain existence \& uniqueness of $\!(\vec{U}^0,\!\!\vec{V}^0)$. See Lemma \ref{l: simple} and Lemma \ref{l: v0-yindep}
    \item[c.] Obtain the defining system of $(\vec{U}^0,\vec{V}^0)$. See equations (\ref{eq: cellV1})-(\ref{eq: upscaledV}) and Lemma \ref{l: u0v0-sys}.
    \item[d.] Statement of the upscaled system. See Theorem \ref{t: it-is-u0}.
    \end{itemize}
\end{itemize}

\subsection{$\epsilon$-independent bounds for $(\vec{U}^\epsilon,\vec{V}^\epsilon)$}\label{H: sec: 4}
In this section, we show $\epsilon$-independent bounds for a weak solution $(\vec{U}^\epsilon,\vec{V}^\epsilon)$ to the Neumann problem (\ref{eq: Aeps-sys})-(\ref{eq: Vhole=0}). We define a weak solution to the Neumann problem (\ref{eq: Aeps-sys})-(\ref{eq: Vhole=0}) as a pair $(\vec{U}^\epsilon,\vec{V}^\epsilon)\in H^1((0,T)\times\Omega^\epsilon)^N\times L^\infty((0,T),\mathbb{V}_\epsilon)^N$ satisfying
\begin{equation*}
    \text{\textbf{(P$_w^\epsilon$)}}\qquad\qquad\begin{dcases}
    \int_{\Omega^\epsilon}\vec{\phi}^\top\left[\tens{M}^\epsilon\vec{V}^\epsilon-\vec{H}^\epsilon - \tens{K}^\epsilon\vec{U}^\epsilon-\tens{J}^\epsilon\cdot\nabla\vec{U}^\epsilon\right]&\cr
    \qquad\qquad+(\nabla\vec{\phi})^\top\cdot\left(\tens{E}^\epsilon\cdot\nabla\vec{V}^\epsilon+\tens{D}^\epsilon\vec{V}^\epsilon\right)\mathrm{d}\vec{x} = 0,\cr
    \int_{\Omega^\epsilon}\vec{\psi}^\top\left[\frac{\partial\vec{U}^\epsilon}{\partial t}+\tens{L}\vec{U}^\epsilon- \tens{G}\vec{V}^\epsilon\right]\mathrm{d}\vec{x} = 0,\cr
    \vec{U}^\epsilon(0,\vec{x})=\vec{U}^*(\vec{x}) \text{ for all }\vec{x}\in\overline{\Omega^\epsilon},
    \end{dcases}
\end{equation*}
for a.e. $t\in(0,T)$ and for all test-functions $\vec{\phi}\in \mathbb{V}_\epsilon^N$ and $\vec{\psi}\in L^2(\Omega^\epsilon)^N$.\\
The existence and uniqueness of solutions to system (P$_w^\epsilon$) can only hold when the parameters are well-balanced. The next lemma provides a set of parameters for which these parameters are well-balanced.
\begin{lemma}\label{H: l: : apriori} Assume assumptions (A1)-(A3) hold and we have $\epsilon\in(0,\epsilon_0)$ for $\epsilon_0>0$, then there exist positive constants $\tilde{m}_\alpha$, $\tilde{e}_i$, $\tilde{H}$, $\tilde{K}_\alpha$, $\tilde{J}_{i\alpha}$,  for $\alpha\in\{1,\ldots,N\}$ and $i\in\{1,\ldots,d\}$ such that the \textit{a priori} estimate
    \begin{multline}
        \sum_{\alpha=1}^N\tilde{m}_\alpha\|V^\epsilon_\alpha\|^2_{L^2(\Omega)}+\sum_{i=1}^d\sum_{\alpha=1}^N\tilde{e}_i\left\|\frac{\mathrm{d}V^\epsilon_\alpha}{\mathrm{d}x_i}\right\|^2_{L^2(\Omega)}\\
        \leq \tilde{H}+\sum_{\alpha=1}^N\tilde{K}_\alpha\|U^\epsilon_\alpha\|^2_{L^2(\Omega)}+\sum_{i=1}^d\sum_{\alpha=1}^N\tilde{J}_{i\alpha}\left\|\frac{\mathrm{d}U^\epsilon_\alpha}{\mathrm{d}x_i}\right\|^2_{L^2(\Omega)}\label{H: eq:  aprioriV}
    \end{multline}
    holds for a.e. $t\in(0,T)$.
\end{lemma}
\begin{proof}
We test the first equation of \textbf{(P$_w^\epsilon$)} with $\vec{\phi} = \vec{V}^\epsilon$ and apply Young's inequality wherever a product is not a square. A non-square product containing both  $\vec{U}^\epsilon$ and $\nabla\vec{V}^\epsilon$ can only be found in the $\tens{D}$-term. Hence, Young's inequality allows all other non-square product terms to have a negligible effect on the coercivity constants $m_\alpha$ and $e_i$, while affecting $\tilde{H}$, $\tilde{K}_\alpha$, $\tilde{J}_{i\alpha}$. Therefore, we only need to enforce two inequalities to prove the lemma by guaranteeing coercivity, i.e.
\begin{subequations}
\begin{align}
    e_i-\sum_{\alpha=1}^N\frac{\eta_{i\beta\alpha}}{2}\tilde{D}_{i\beta\alpha}\geq\tilde{e}_i>0&\;\text{for }\beta\in\{1,\ldots,N\},\,i\in\{1,\ldots,d\},\label{eq: e-coerc}\\
    m_\alpha-\sum_{i=1}^d\sum_{\beta=1}^N\frac{\tilde{D}_{i\beta\alpha}}{2\eta_{i\beta\alpha}}\geq\tilde{m}_\alpha>0&\;\text{for }\alpha\in\{1,\ldots,N\},\label{eq: m-coerc}
\end{align}
\end{subequations}
where $\tilde{D}_{i\beta\alpha}=\|D_{i\beta\alpha}\|_{L^\infty(\mathbf{R}_+\times\Omega;C_\#(Y^*))}$.
We can choose $\eta_{i\beta\alpha}>0$ satisfying
\begin{equation}
    \frac{dN\tilde{D}_{i\beta\alpha}}{2m_\alpha}<\eta_{i\beta\alpha}<\frac{2e_i}{N\tilde{D}_{i\beta\alpha}},\label{eq: coerc-bound}
\end{equation}
if inequality (\ref{H: eq:  ineq}) in assumption (A3) is satisfied. For the exact definition of the constants $\tilde{m}_\alpha$, $\tilde{e}_i$, $\tilde{H}$, $\tilde{K}_\alpha$, $\tilde{J}_{i\alpha}$, see equations (\ref{eq: app tildembound})-(\ref{eq: app tildehibound}) in Appendix \ref{app: constants}.
\qed\end{proof}
\begin{theorem}\label{H: t: existuniq} Assume (A1)-(A3) to hold, then there exist positive constants $C$, $\tilde{\kappa}$ and $\lambda$ independent of $\epsilon$ such that
\begin{equation}\label{eq: timebound}
    \|\vec{U}^\epsilon\|_{H^1(\Omega^\epsilon)^N}(t)\leq Ce^{\lambda t},\qquad \|\vec{V}^\epsilon\|_{\mathbb{V}_\epsilon^N}(t)\leq C(1+\tilde{\kappa}e^{\lambda t})
\end{equation}
   hold for $t\geq0$.
\end{theorem}
\begin{proof}
By (A1) - (A3) there exist positive numbers $\tilde{m}_\alpha$, $\tilde{e}_i$, $\tilde{H}$, $\tilde{K}_\alpha$, $\tilde{J}_{i\alpha}$ for $\alpha\in\{1,\ldots,N\}$ and $i\in\{1,\ldots,d\}$ such that the \textit{a priori} estimate (\ref{H: eq:  aprioriV}) stated in Lemma \ref{H: l: : apriori} holds.
Moreover, what concerns system \textbf{(P$_w^\epsilon$)} there exist $L_G$, $L_N$, $G_G$, and $G_N$, see equations (\ref{eq: app LNbound})-(\ref{eq: app GGbound}) in Appendix \ref{app: constants}, such that
\begin{subequations}
\begin{align}
    \frac{\partial}{\partial t}\|\vec{U}^\epsilon\|_{L^2(\Omega^\epsilon)^N}^2&\leq L_N\|\vec{U}^\epsilon\|_{L^2(\Omega^\epsilon)^N}^2+G_N\|\vec{V}^\epsilon\|_{L^2(\Omega^\epsilon)^N}^2,\label{eq: L1bound}\\
    \frac{\partial}{\partial t}\|\nabla \vec{U}^\epsilon\|_{L^2(\Omega^\epsilon)^{d\times N}}^2&\leq L_G\|\vec{U}^\epsilon\|_{L^2(\Omega^\epsilon)^N}^2+L_N\|\nabla \vec{U}^\epsilon\|_{L^2(\Omega^\epsilon)^{d\times N}}^2\cr
    &\quad+G_G\|\vec{V}^\epsilon\|_{L^2(\Omega^\epsilon)^N}^2+G_N\|\nabla \vec{V}^\epsilon\|_{L^2(\Omega^\epsilon)^{d\times N}}^2\label{eq: L2bound}
\end{align}
\end{subequations}
hold. Adding (\ref{eq: L1bound}) and (\ref{eq: L2bound}), and using (\ref{H: eq:  aprioriV}), we obtain a positive constant $I$ and a vector $\vec{J}\in\mathbf{R}_+^N$ such that
\begin{equation}
    \frac{\partial}{\partial t}\|\vec{U}^\epsilon\|_{H^1(\Omega^\epsilon)^N}^2\leq \vec{J}+I\|\vec{U}^\epsilon\|_{H^1(\Omega^\epsilon)^N}^2\label{H: eq:  ineq-gron}
\end{equation}
with
\begin{subequations}
\begin{align}
    I &=\max\!\left\{\!0,\! L_N\!+\!\max\!\left\{\!L_G\!+\!G_M\!\!\!\!\max_{1\leq\alpha\leq N}\!\{\tilde{K}_\alpha\},G_M\!\!\!\!\!\!\!\!\!\!\!\max_{1\leq\alpha\leq N,1\leq i\leq d}\{\tilde{J}_{i\alpha}\}\!\right\}\!\right\},\label{eq: Ibound}\\
    G_M &= \max_{1\leq\alpha<N,1\leq i\leq d}\left\{\frac{G_N+G_G}{\tilde{m}_\alpha},\frac{G_N}{\tilde{e}_i}\right\}.\label{eq: GMbound}
    \end{align}
\end{subequations}
Applying Gronwall's inequality, see \cite[Thm. 1]{Dragomir2003}, to (\ref{H: eq:  ineq-gron}) yields the existence of a constant $\lambda$ defined as $\lambda = I/2$, such that
\begin{equation}\label{eq: timebound}
    \|\vec{U}^\epsilon\|_{H^1(\Omega^\epsilon)^N}(t)\leq Ce^{\lambda t},\qquad \|\vec{V}^\epsilon\|_{\mathbb{V}_\epsilon^N}(t)\leq C(1+\tilde{\kappa}e^{\lambda t})
\end{equation}
with $\tilde{\kappa} = \max_{1\leq\alpha\leq N,1\leq i\leq d}\{\tilde{K}_\alpha,\tilde{J}_{i\alpha}\}$.
\qed\end{proof}
\begin{remarkArt}\label{r: lambda}
It is difficult to obtain exact expressions for optimal values of $L_N$, $L_G$, $G_N$ and $G_G$ such that a minimal positive value of $\lambda$ is obtained. See Appendix \ref{app: constants} for the exact dependence of $\lambda$ on the parameters involved in the Neumann problem (\ref{eq: Aeps-sys})-(\ref{eq: Vhole=0}).
\end{remarkArt}
\begin{remarkArt}\label{r: measurability}
The $(0,T)\times\Omega^\epsilon$-measurability of $\vec{U}^\epsilon$ and $\vec{V}^\epsilon$ can be proven based on the Rothe-method (discretization in time) in combination with the convergence of piecewise linear functions to any function in the spaces $H^1((0,T)\times\Omega^\epsilon)$ or $L^\infty((0,T);\mathbb{V}_\epsilon)$. One can prove that both $\vec{U}^\epsilon$ and $\vec{V}^\epsilon$ are measurable and are weak solutions to \textbf{($\mathbf{P}_w^\epsilon$)}. See Chapter 2 in \cite{VromansLIC} for a pseudo-parabolic system for which the Rothe-method is used to show existence (and hence also measurability).
\end{remarkArt}
\begin{remarkArt}\label{r: gradientL}
Since we have $\tens{G}\!\!\in\!\! L^\infty(\mathbf{R}_+;\!W^{1,\infty}(\Omega))^{N\!\times\! N}$ and $\vec{V}^\epsilon\!\!\in\!\! L^\infty((0,T);\!\mathbb{V}_\epsilon)^N$, we are allowed to differentiate equation (\ref{eq: L-sys}) with respect to $\vec{x}$ and test the resulting identity with both $\nabla\vec{U}^\epsilon$ and $\frac{\partial}{\partial t}\nabla\vec{U}^\epsilon$. However, conversely, we are not allowed to differentiate   equation (\ref{eq: Aeps-sys}) with respect to $t$ as all tensors have insufficient regularity: they are in $L^\infty(\mathbf{R}_+\times\Omega^\epsilon)^{N\times N}$.
\end{remarkArt}
\begin{remarkArt}\label{r: J=0}
We cannot differentiate equation (\ref{eq: L-sys}) with respect to $\vec{x}$ when $\tens{L}$ or $\tens{G}$ has decreased spatial regularity, for example
$L^\infty((0,T)\times\Omega)^{N\times N}$. One can still obtain unique solutions of \textbf{($\mathbf{P}_w^\epsilon$)} if and only if $\tens{J}^\epsilon = \tens{0}$ holds, since it removes the $\nabla\vec{U}^\epsilon$ term from equation (\ref{eq: Aeps-sys}). Consequently, Theorem \ref{H: t: existuniq} holds with $\vec{U}^\epsilon\in H^1((0,T);L^2(\Omega^\epsilon))$ and $\tens{J}^\epsilon=\tens{0}$ under the additional relaxed regularity assumption $\tens{L},\tens{G}\in L^\infty((0,T)\times\Omega)^{N\times N}$ and with $\lambda$ modified by taking $L_G=\tilde{J}_{i\alpha}=0$ and by replacing $G_M$ with $G_N/\min_{1\leq\alpha\leq N}\tilde{m}_\alpha$.
\end{remarkArt}
\subsection{Upscaling the system \textbf{(P$^\epsilon_w$)} via two-scale convergence}
\label{H: sec: 5}
We recall the notation $\hat{f}^\epsilon$ to denote the extension on $\Omega$ via the operator $\mathcal{P}^\epsilon$ for $f^\epsilon$ defined on $\Omega^\epsilon$. This extension operator $\mathcal{P}^\epsilon$, as defined in Theorem \ref{t: extension}, is well-defined if both $\partial\mathcal{T}$ and $\partial \Omega$ are $C^2$-regular, assumption (A4) holds, and $\partial\mathcal{T}\cap\partial Y=\emptyset$. Hence, the extension operator is well-defined in our setting.
\begin{lemma}
    \label{H: l: : two-scale-conv} Assume (A1)-(A4) to hold. For each $\epsilon\in(0,\epsilon_0)$, let the pair of sequences $(\vec{U}^\epsilon,\vec{V}^\epsilon)\in H^1((0,T)\times\Omega^\epsilon)^N\times L^\infty((0,T);\mathbb{V}_\epsilon)^N$ be the unique weak solution to \textbf{($\mathbf{P}^{\,\epsilon}_w$)}. Then this sequence of weak solutions satisfies the estimates
\begin{equation}
    \|\vec{U}^\epsilon\|_{H^1((0,T)\times\Omega^\epsilon)^N}+\|\vec{V}^\epsilon\|_{L^\infty((0,T);\mathbb{V}_\epsilon)^N}\leq C,\label{H: eq:  bounds}
\end{equation}
for all $\epsilon\in(0,\epsilon_0)$ and there exist vector functions
\begin{subequations}
\begin{align}
    \vec{u}&\text{ in }H^1((0,T)\times\Omega)^N,\label{eq: 2u1}\\
    \mathcal{U}&\text{ in }H^1((0,T);L^2(\Omega;H^1_\#(Y^*)/\mathbf{R}))^N,\label{eq: 2u2}\\
    \vec{v}&\text{ in }L^\infty((0,T);H^1_0(\Omega))^N,\label{eq: 2v1}\\
    \mathcal{V}&\text{ in }L^\infty((0,T)\times\Omega;H^1_\#(Y^*)/\mathbf{R})^N,\label{eq: 2v2}
\end{align}
\end{subequations}
and a subsequence $\epsilon'\subset\epsilon$, for which the following two-scale convergences
\begin{subequations}
\begin{align}
    \hat{\vec{U}}^{\epsilon'}&\overset{2}{\longrightarrow}\vec{u}(t,\vec{x}),\label{H: eq:  twoscale1}\\
    \frac{\partial}{\partial t}\hat{\vec{U}}^{\epsilon'}&\overset{2}{\longrightarrow}\frac{\partial}{\partial t}\vec{u}(t,\vec{x}),\label{H: eq:  twoscale1t}\\
    \nabla\hat{\vec{U}}^{\epsilon'}&\overset{2}{\longrightarrow}\nabla\vec{u}(t,\vec{x})+\nabla_{\vec{y}}\mathcal{U}(t,\vec{x},\vec{y}),\label{H: eq:  twoscale1D}\\
    \frac{\partial}{\partial t}\nabla\hat{\vec{U}}^{\epsilon'}&\overset{2}{\longrightarrow}\frac{\partial}{\partial t}\nabla\vec{u}(t,\vec{x})+\frac{\partial}{\partial t}\nabla_{\vec{y}}\mathcal{U}(t,\vec{x},\vec{y}),\label{H: eq:  twoscale1Dt}\\
    \hat{\vec{V}}^{\epsilon'}&\overset{2}{\longrightarrow}\vec{v}(t,\vec{x}),\label{H: eq:  twoscale2}\\
    \nabla\hat{\vec{V}}^{\epsilon'}&\overset{2}{\longrightarrow}\nabla\vec{v}(t,\vec{x})+\nabla_{\vec{y}}\mathcal{V}(t,\vec{x},\vec{y})\label{H: eq:  twoscale2D}
\end{align}
\end{subequations}
 hold for a.e. $t\in(0,T)$, $\vec{x}\in\Omega$, and $\vec{y}\in Y^*$.
\end{lemma}
\begin{proof}
For all $\epsilon>0$, Theorem \ref{H: t: existuniq} gives the bounds (\ref{H: eq:  bounds}) independent of the choice of $\epsilon$. Hence, $\hat{\vec{U}}^\epsilon\rightharpoonup \vec{u}$ in $H^1((0,T)\times\Omega)^N$ and $\hat{\vec{V}}^\epsilon\rightharpoonup \vec{v}$ in $L^\infty((0,T);H^1_0(\Omega))^N$ as $\epsilon\rightarrow0$. By Proposition \ref{H: p: grad-extra} in Appendix \ref{a: 2-scale}, we obtain a subsequence $\epsilon'\subset\epsilon$ and functions $\vec{u}\in H^1((0,T)\times\Omega)^N$, $\vec{v}\in L^2((0,T);H^1_0(\Omega))^N$, $\mathcal{U},\mathcal{V}\in L^2((0,T)\times\Omega;H^1_\#(Y^*)/\mathbf{R})^N$ such that (\ref{H: eq:  twoscale1}), (\ref{H: eq:  twoscale1t}), (\ref{H: eq:  twoscale1D}), (\ref{H: eq:  twoscale2}), and (\ref{H: eq:  twoscale2D}) hold for a.e. $t\in(0,T)$. Moreover, there exists a vector function $\tilde{\mathcal{U}}\in L^2((0,T)\times\Omega;H^1_\#(Y^*)/\mathbf{R})^N$ such that the following two-scale convergence
\begin{equation}
    \frac{\partial}{\partial t}\nabla\hat{\vec{U}}^{\epsilon'}\overset{2}{\longrightarrow}\frac{\partial}{\partial t}\nabla\vec{u}(t,\vec{x})+\nabla_{\vec{y}}\tilde{\mathcal{U}}(t,\vec{x},\vec{y})\label{eq: 2u3}
\end{equation}
holds for the same subsequence $\epsilon'$. Using two-scale convergence, Fubini's Theorem and partial integration in time, we obtain an increased regularity for $\mathcal{U}$, i.e. $\mathcal{U}\in H^1((0,T);L^2(\Omega;H^1_\#(Y^*)/\mathbf{R}))^N$, with $\frac{\partial}{\partial t}\nabla_{\vec{y}}\mathcal{U} = \nabla_{\vec{y}}\tilde{\mathcal{U}}$.
\qed\end{proof}
By Lemma \ref{H: l: : two-scale-conv}, we can determine what the macroscopic version of \textbf{(P$^\epsilon_w$)}, which we denote by \textbf{(P$^0_w$)}. This is as stated in Theorem \ref{H: t: upscale}.
\begin{theorem}\label{H: t: upscale} Assume the hypotheses of Lemma \ref{H: l: : two-scale-conv} to be satisfied. Then the two-scale limits $\vec{u}\in H^1((0,T)\times\Omega)^N$ and $\vec{v}\in L^\infty((0,T);H^1_0(\Omega))^N$ introduced in Lemma \ref{H: l: : two-scale-conv} form a weak solution to \begin{equation*}
    \textbf{$(\mathbf{P}^0_w)$}\qquad\quad\begin{dcases}
    \int_\Omega\vec{\phi}^\top\left[\overline{\tens{M}}\vec{v}-\overline{\vec{H}} - \overline{\tens{K}}\vec{u}\right]&\cr
    \qquad\qquad+(\nabla\vec{\phi})^\top\cdot\left(\tens{E}^*\cdot\nabla\vec{v}+\tens{D}^*\vec{v}\right)\mathrm{d}\vec{x} = 0,\cr
    \int_\Omega\vec{\psi}^\top\left[\frac{\partial \vec{u}}{\partial t}+\tens{L}\vec{u}- \tens{G}\vec{v}\right]\mathrm{d}\vec{x} = 0,\cr
    \vec{u}(0,\vec{x})=\vec{U}^*(\vec{x})\qquad\text{for }\vec{x}\in\Omega,
    \end{dcases}
\end{equation*}
for a.e. $t\in(0,T)$ for all test functions $\vec{\phi}\in H^1_0(\Omega)^N$, and $\vec{\psi}\in L^2(\Omega)^N$, where the barred tensors and vectors are $Y^*$ averaged functions as introduced in (A2). Furthermore,
\begin{equation}
\tens{E}^*=\frac{1}{|Y|}\int_{Y^*}\tens{E}\cdot(\tens{1}+\nabla_{\vec{y}}\vec{W})\mathrm{d}\vec{y},\quad\tens{D}^*=\frac{1}{|Y|}\int_{Y^*}\tens{D}+\tens{E}\cdot\nabla_{\vec{y}}\tens{Z}\mathrm{d}\vec{y}\label{eq: *matrices}
\end{equation}
are the wanted \emph{effective} coefficients. The auxiliary tensors\\
$\tens{Z}_{\alpha\beta},W_i\in L^\infty(0,T;W^{2,\infty}(\Omega;H^1_\#(Y^*)/\mathbf{R}))$ satisfy the cell problems
\begin{subequations}
\begin{align}
\vec{0}&=\int_{Y^*}\vec{\Phi}^\top\cdot(\nabla_{\vec{y}}\cdot\left[\tens{E}\cdot(\tens{1}+\nabla_{\vec{y}}\vec{W})\right])\mathrm{d}\vec{y}=\int_{Y^*}\vec{\Phi}^\top\cdot(\nabla_{\vec{y}}\cdot\hat{\tens{E}})\mathrm{d}\vec{y},\label{H: eq:  cell1}\\
\vec{0}&=\int_{Y^*}\vec{\Psi}^\top(\nabla_{\vec{y}}\cdot\left[\tens{D}+\tens{E}\cdot\nabla_{\vec{y}}\tens{Z}\right])\mathrm{d}\vec{y}=\int_{Y^*}\vec{\Psi}^\top(\nabla_{\vec{y}}\cdot\hat{\tens{D}})\mathrm{d}\vec{y}\label{H: eq:  cell2}
\end{align}
\end{subequations}
for all $\vec{\Phi}\in C_\#(Y^*)^d$, $\vec{\Psi}\in C_\#(Y^*)^N$.
\end{theorem}
\begin{proof} The solution to system \textbf{(P$^\epsilon_w$)} is extended to $\Omega$ by taking $\hat{\vec{H}}^\epsilon$, $\hat{\vec{V}}^\epsilon$, $\hat{\vec{U}}^\epsilon$ for $\vec{H}^\epsilon$, $\vec{V}^\epsilon$, $\vec{U}^\epsilon$, respectively. The extended system is satisfied on $\mathcal{T}^\epsilon\cap\Omega$ and it satisfies the boundary conditions on $\partial_{int}\Omega^\epsilon$ of system \textbf{(P$^\epsilon_w$)}. Hence, it is sufficient to look at \textbf{(P$^\epsilon_w$)} only. In \textbf{(P$^\epsilon_w$)}, we choose $\vec{\psi}=\vec{\psi}^\epsilon = \vec{\Psi}\left(t,\vec{x},\frac{\vec{x}}{\epsilon}\right)$ for the test function $\vec{\Psi}\in L^2((0,T);\mathcal{D}(\Omega^\epsilon;C^\infty_\#(Y^*)))^N$, $\vec{\phi}=\vec{\phi}^\epsilon = \vec{\Phi}(t,\vec{x})+\epsilon\vec{\varphi}\left(t,\vec{x},\frac{\vec{x}}{\epsilon}\right)$ for the test functions $\vec{\Phi}\in L^2((0,T);C^\infty_0(\Omega^\epsilon))^N$,  $\vec{\varphi}\in L^2((0,T);\mathcal{D}(\Omega^\epsilon;C^\infty_\#(Y^*)))^N$. Corollary \ref{H: c: 2-scale} and Theorem \ref{H: t: tripprod} in combination with (\ref{H: eq:  correspondence}) lead to $\tens{T}^\epsilon\overset{2}{\longrightarrow}\tens{T}$, where $\tens{T}^\epsilon$ is an arbitrary tensor or vector in (P$^\epsilon_w$) other than $\tens{L}$ and $\tens{G}$. Moreover, by Corollary \ref{H: c: 2-scale} and Propositions \ref{H: p: grad-extra} and \ref{H: p: grad} we have $\vec{\psi}^\epsilon\overset{2}{\longrightarrow}\vec{\Psi}(t,\vec{x},\vec{y})$,  
$\vec{\phi}^\epsilon\overset{2}{\longrightarrow}\vec{\Phi}(t,\vec{x})$, and
$\nabla\vec{\phi}^\epsilon\overset{2}{\longrightarrow}\nabla\vec{\Phi}(t,\vec{x})+\nabla_{\vec{y}}\vec{\varphi}(t,\vec{x},\vec{y})$. By Corollary \ref{H: c: 2-scale} and Theorem \ref{H: t: tripprod}, there is a two-scale limit of \textbf{(P$^\epsilon_w$)}, reading
\begin{multline}
    \int_\Omega\frac{1}{|Y|}\int_{Y^*}
    \vec{\Phi}^\top\left[\tens{M}\vec{v}-\vec{H}-\tens{K}\vec{u}\right]\\
    +(\nabla\vec{\Phi}+\nabla_{\vec{y}}\vec{\varphi})^\top\cdot\left[\tens{E}\cdot(\nabla\vec{v}+\nabla_{\vec{y}}\mathcal{V})+\tens{D}\vec{v}\right]\\ 
    +\vec{\Psi}^\top\left[\frac{\partial\vec{u}}{\partial t}+\tens{L}\vec{u}-\tens{G}\vec{v}\right]\mathrm{d}\vec{y}\mathrm{d}\vec{x}=0.\label{eq: 2scalelimsys}
\end{multline}
Similarly, the initial condition
\begin{equation}
\vec{u}(0,\vec{x})=\vec{U}^*(\vec{x}), \quad \vec{x}\in\overline{\Omega},\label{eq: 2scaleinit}
\end{equation}
is satisfied by $\vec{u}$ as $\nabla_{\vec{y}}\vec{u}=\tens{0}$ holds.\\
For $\vec{\Phi}=\vec{\Psi} = \vec{0}$, we can take $\mathcal{V} = \vec{W}\cdot\nabla\vec{v}+\tens{Z}\vec{v}+\tilde{\mathcal{V}}$, where $\vec{W}$ and $\tens{Z}$ satisfy the cell problems (\ref{H: eq:  cell1}) and (\ref{H: eq:  cell2}), respectively, and $\nabla_{\vec{y}}\tilde{\mathcal{V}}=\tens{0}$. Moreover, we obtain $\vec{v}\in L^\infty((0,T);H^2(\Omega))$ due to (A1). Then Proposition \ref{H: p: grad-extra}, Theorem \ref{H: t: tripprod} and the embedding $H^{1/2}(Y^*)\hookrightarrow L^2(\partial\mathcal{T})$ yields $0 = \frac{\partial \vec{V}^\epsilon}{\partial\nu_{\tens{D}^\epsilon}}\overset{2}{\longrightarrow}(\hat{E}\nabla\vec{v}+\hat{D}\vec{v})\cdot\vec{n} = 0$ on $\partial Y^*$, which is automatically guaranteed by (\ref{H: eq:  cell1}) and (\ref{H: eq:  cell2}).
\qed\end{proof}
Hence, \textbf{(P$^0_w$)} yields the strong form system
\begin{equation*}
    \text{\textbf{(P$^0_s$)}}\quad\begin{dcases}
    \overline{\tens{M}}\vec{v}-\nabla\cdot\left(\tens{E}^*\cdot\nabla\vec{v}+\tens{D}^*\vec{v}\right) = \overline{\vec{H}} + \overline{\tens{K}}\vec{u}&\text{in }(0,T)\times\Omega,\cr
    \frac{\partial\vec{u}}{\partial t}+\tens{L}\vec{u}= \tens{G}\vec{v}&\text{in }(0,T)\times\Omega,\cr
    \vec{v}= \vec{0}&\text{on }(0,T)\times\partial\Omega,\cr
    \vec{u}= \vec{U}^*&\text{on }\{0\}\times\overline{\Omega},
    \end{dcases}\label{H: eq:  system-Ps}
\end{equation*}
when, next to the regularity of (A1), the following regularity holds:
    \begin{subequations}
    \begin{align}
     M_{\alpha\beta},H_\alpha,K_{\alpha\beta}&\in C(0,T;C^{1}(\Omega;C^1_\#(Y^*))),\label{eq: regularity1a}\\
     E_{ij},D_{i\alpha\beta}&\in C(0,T;C^{2}(\Omega;C^2_\#(Y^*))),\label{eq: regularity1b}\\
     L_{\alpha\beta},G_{\alpha\beta}&\in C(0,T;C^{1}(\Omega)),\label{eq: regularity1c}\\
     \vec{U}^*&\in C(\overline{\Omega}),\label{eq: regularity1d}
    \end{align}
    \end{subequations}
for all $T\in\mathbf{R}_+$, when both $\partial\Omega$ and $\partial\mathcal{T}$ are $C^3$-boundaries.
\subsection{Upscaling via asymptotic expansions} \label{sec: 4.3}
Even though the previous section showed that there is a two-scale limit $(\vec{u},\vec{v})$, it is necessary to show the relation between $(\vec{u},\vec{v})$ and $(\vec{U}^\epsilon,\vec{V}^\epsilon)$. To this end, we first rewrite the Neumann problem (\ref{eq: Aeps-sys})-(\ref{eq: Vhole=0}) and then use asymptotic expansions such that we are lead to the two-scale limit, including the cell-functions, in a natural way.\\
$\;$\\
The Neumann problem (\ref{eq: Aeps-sys})-(\ref{eq: Vhole=0}) can be written in operator form as
\begin{equation}\label{eq: neumann-operator}
    \begin{dcases}
    \mathcal{A}^\epsilon \vec{V}^\epsilon = \mathcal{H}^\epsilon\vec{U}^\epsilon&\text{ on }(0,T)\times\Omega^\epsilon,\cr
    \mathcal{L}\vec{U}^\epsilon = \tens{G}\vec{V}^\epsilon&\text{ on }(0,T)\times\Omega^\epsilon,\cr
    \vec{U}^\epsilon = \vec{U}^*&\text{ in }\{0\}\times\Omega^\epsilon,\cr
    \vec{V}^\epsilon = 0&\text{ on }(0,T)\times\partial_{ext}\Omega^\epsilon,\cr
    \frac{\mathrm{d} \vec{V}^\epsilon}{\mathrm{d}\nu_{\tens{D}^\epsilon}} = 0&\text{ on }(0,T)\times\partial_{int}\Omega^\epsilon.
    \end{dcases}
\end{equation}
as indicated in Section \ref{s: sec2}.\\.
We postulate the following asymptotic expansions in $\epsilon$ of $\vec{U}^\epsilon$ and $\vec{V}^\epsilon$:
\begin{subequations}
\begin{align}
    \vec{V}^\epsilon(t,\vec{x}) &= \vec{V}^0\left(t,\vec{x},\frac{\vec{x}}{\epsilon}\right)+\epsilon \vec{V}^1\left(t,\vec{x},\frac{\vec{x}}{\epsilon}\right)+\epsilon^2 \vec{V}^2\left(t,\vec{x},\frac{\vec{x}}{\epsilon}\right)+\cdots, \label{eq: formalV}\\
    \vec{U}^\epsilon(t,\vec{x}) &= \vec{U}^0\left(t,\vec{x},\frac{\vec{x}}{\epsilon}\right)+\epsilon \vec{U}^1\left(t,\vec{x},\frac{\vec{x}}{\epsilon}\right)+\epsilon^2 \vec{U}^2\left(t,\vec{x},\frac{\vec{x}}{\epsilon}\right)+\cdots\label{eq: formalU}
\end{align}
\end{subequations}
Let $\vec{\Phi} = \vec{\Phi}(t,\vec{x},\vec{y})\in L^\infty(0,T;C^2(\Omega;C^2_\#(Y^*)))^N$ be a vector function depending on two spatial variables $\vec{x}$ and $\vec{y}$, and introduce $\vec{\Phi}^\epsilon(t,\vec{x}) = \vec{\Phi}(t,\vec{x},\vec{x}/\epsilon)$. Then the total spatial derivatives in $\vec{x}$ become two partial derivatives, one in $\vec{x}$ and one in $\vec{y}$:
\begin{subequations}
\begin{align}
    \nabla\vec{\Phi}^\epsilon(t,\vec{x}) &= \frac{1}{\epsilon}(\nabla_{\vec{y}}\vec{\Phi})\left(t,\vec{x},\frac{\vec{x}}{\epsilon}\right)+ (\nabla_{\vec{x}}\vec{\Phi})\left(t,\vec{x},\frac{\vec{x}}{\epsilon}\right),\label{eq: gradexpans}\\
    \nabla\cdot\vec{\Phi}^\epsilon(t,\vec{x}) &= \frac{1}{\epsilon}(\nabla_{\vec{y}}\cdot\vec{\Phi})\left(t,\vec{x},\frac{\vec{x}}{\epsilon}\right)+ (\nabla_{\vec{x}}\cdot\vec{\Phi})\left(t,\vec{x},\frac{\vec{x}}{\epsilon}\right).\label{eq: divexpans}
\end{align}
\end{subequations}
Do note, the evaluation $\vec{y}=\vec{x}/\epsilon$ is suspended as is common in formal asymptotic expansions, leading to the use of $\vec{y}\in Y^*$ and $\vec{x}\in\Omega$.\\
Hence, $\mathcal{A}^\epsilon\vec{\Phi}^\epsilon$ can be formally expanded:
\begin{equation}\label{eq: Aformal}
    \mathcal{A}^\epsilon\vec{\Phi}^\epsilon = \left[\left(\frac{1}{\epsilon^2}\mathcal{A}^0+\frac{1}{\epsilon}\mathcal{A}^1+\mathcal{A}^2\right)\vec{\Phi}\right]\left(t,\vec{x},\frac{\vec{x}}{\epsilon}\right),
\end{equation}
where
\begin{subequations}
\begin{align}
    \mathcal{A}^0\vec{\Phi}&=-\nabla_{\vec{y}}\cdot\left(\tens{E}\cdot\nabla_{\vec{y}}\vec{\Phi}\right),\label{eq: A0}\\
    \mathcal{A}^1\vec{\Phi}&=-\nabla_{\vec{y}}\cdot\left(\tens{E}\cdot\nabla_{\vec{x}}\vec{\Phi}\right)-\nabla_{\vec{x}}\cdot\left(\tens{E}\cdot\nabla_{\vec{y}}\vec{\Phi}\right)-\nabla_{\vec{y}}\cdot\left(\tens{D}\vec{\Phi}\right),\label{eq: A1}\\
    \mathcal{A}^2\vec{\Phi}&=\!\qquad\!\qquad\!\qquad\tens{M}\vec{\Phi}-\nabla_{\vec{x}}\cdot\left(\tens{E}\cdot\nabla_{\vec{x}}\vec{\Phi}\right)-\nabla_{\vec{x}}\cdot\left(\tens{D}\vec{\Phi}\right).\label{eq: A2}
\end{align}
\end{subequations}
Moreover, $\mathcal{H}^\epsilon\vec{\Phi}^\epsilon$ can be written as $\vec{H}+(\mathcal{H}^0+\epsilon \mathcal{H}^1)\vec{\Phi}$, where
\begin{subequations}
\begin{align}
    \mathcal{H}^0&=\tens{K}+\tens{J}\cdot\nabla_{\vec{y}},\label{eq: H0}\\
    \mathcal{H}^1&=\qquad\tens{J}\cdot\nabla_{\vec{x}}.\label{eq: H1}
\end{align}
\end{subequations}
Since the outward normal $\vec{n}$ on $\partial\mathcal{T}$ depends only on $\vec{y}$ and the outward normal $\vec{n}^\epsilon$ on $\partial_{int}\Omega^\epsilon=\partial\mathcal{T}^\epsilon\cap\Omega$ is defined as the $Y$-periodic function $\left.\vec{n}\right|_{\vec{y}=\vec{x}/\epsilon}$, one has
\begin{multline}\label{eq: boundformal}
    \frac{\partial\vec{\Phi}^\epsilon}{\partial\nu_{\tens{D}^\epsilon}} =\left(\tens{E}^\epsilon\cdot\frac{\mathrm{d}\vec{\Phi}^\epsilon}{\mathrm{d}\vec{x}}+\tens{D}^\epsilon\vec{\Phi}^\epsilon\right)\cdot\vec{n}^\epsilon\\
    =\left(\frac{1}{\epsilon}\tens{E}\cdot\nabla_{\vec{y}}\vec{\Phi}+\tens{E}\cdot\nabla_{\vec{x}}\vec{\Phi}+\tens{D}\vec{\Phi}\right)\cdot\vec{n}^\epsilon\qquad\qquad\qquad\qquad\qquad\quad\\
    =: \frac{1}{\epsilon} \frac{\partial\vec{\Phi}^\epsilon}{\partial\nu_{\tens{E}}}+\frac{\partial\vec{\Phi}^\epsilon}{\partial\nu_{\tens{D}}}.\qquad\qquad\qquad\qquad\qquad\qquad\qquad\qquad\qquad\qquad\!
\end{multline}
Inserting (\ref{eq: formalV}), (\ref{eq: formalU}), (\ref{eq: Aformal}) - (\ref{eq: boundformal}) into the Neumann problem (\ref{eq: neumann-operator}) and expanding the full problem into powers of $\epsilon$, we obtain the following auxilliary systems:
\begin{equation}\label{eq: A0sys}
    \begin{dcases}
    \mathcal{A}^0\vec{V}^0=0&\text{ in }(0,T)\times\Omega\times Y^*,\cr
    \frac{\partial \vec{V}^0}{\partial\nu_{\tens{E}}} = 0&\text{ on }(0,T)\times\Omega\times\partial\mathcal{T},\cr
    \vec{V}^0=0&\text{ on }(0,T)\times\partial\Omega\times Y^*,\cr
    \vec{V}^0\quad\text{ $Y$-periodic,}
    \end{dcases}
\end{equation}
\begin{equation}\label{eq: A1sys}
    \begin{dcases}
    \mathcal{A}^0\vec{V}^1=-\mathcal{A}^1 \vec{V}^0&\text{ in }(0,T)\times\Omega\times Y^*,\cr
    \frac{\partial \vec{V}^1}{\partial\nu_{\tens{E}}} = -\frac{\partial \vec{V}^0}{\partial\nu_{\tens{D}}}&\text{ on }(0,T)\times\Omega\times\partial\mathcal{T},\cr
    \vec{V}^1=0&\text{ on }(0,T)\times\partial\Omega\times Y^*,\cr
    \vec{V}^1\quad\text{ $Y$-periodic,}
    \end{dcases}
\end{equation}
\begin{equation}\label{eq: A2sys}
    \begin{dcases}
    \mathcal{A}^0\vec{V}^2=-\mathcal{A}^1 \vec{V}^1-\mathcal{A}^2\vec{V}^0+\vec{H}+\mathcal{H}^0\vec{U}^0&\text{ in }(0,T)\times\Omega\times Y^*,\cr
     \frac{\partial \vec{V}^2}{\partial\nu_{\tens{E}}} = -\frac{\partial \vec{V}^1}{\partial\nu_{\tens{D}}}&\text{ on }(0,T)\times\Omega\times\partial\mathcal{T},\cr
     \vec{V}^2=0&\text{ on }(0,T)\times\partial\Omega\times Y^*,\cr
    \vec{V}^2\quad\text{ $Y$-periodic.}
    \end{dcases}
\end{equation}
For $i\geq3$, we have
\begin{equation}\label{eq: A3sys}
    \begin{dcases}
    \mathcal{A}^0\vec{V}^i=-\mathcal{A}^1\vec{V}^{i-1}-\mathcal{A}^2\vec{V}^{i-2}&\text{ in }(0,T)\times\Omega\times Y^*,\cr
    \qquad\qquad\qquad\qquad\,+\mathcal{H}^0\vec{U}^{i-2}+\mathcal{H}^1\vec{U}^{i-3}\cr
     \frac{\partial \vec{V}^i}{\partial\nu_{\tens{E}}} = -\frac{\partial \vec{V}^{i-1}}{\partial\nu_{\tens{D}}}&\text{ on }(0,T)\times\Omega\times\partial\mathcal{T},\cr
     \vec{V}^i=0&\text{ on }(0,T)\times\partial\Omega\times Y^*,\cr
    \vec{V}^i\quad\text{ $Y$-periodic.}
    \end{dcases}
\end{equation}
Furthermore, we have
\begin{equation}\label{eq: L0sys}
    \begin{dcases}
    \mathcal{L}\vec{U}^0=\tens{G}\vec{V}^0&\text{ in }(0,T)\times\Omega\times Y^*,\cr
    \vec{U}^0 = \vec{U}^*&\text{ in } \{0\}\times\Omega\times Y^*,\cr
    \vec{U}^0\quad\text{ $Y$-periodic,}
    \end{dcases}
\end{equation}
and, for $j\geq1$,
\begin{equation}\label{eq: L1sys}
    \begin{dcases}
    \mathcal{L}\vec{U}^j=\tens{G}\vec{V}^j&\text{ in }(0,T)\times\Omega\times Y^*,\cr
    \vec{U}^j = 0&\text{ in } \{0\}\times\Omega\times Y^*,\cr
    \vec{U}^j\quad\text{ $Y$-periodic.}
    \end{dcases}
\end{equation}
The existence and uniqueness of weak solutions of the systems (\ref{eq: A0sys}) - (\ref{eq: A3sys}) is stated in the following Lemma:
\begin{lemma}\label{l: simple}
Let $F\in L^2(Y^*)$ and $g\in L^2(\partial \mathcal{T})$ be $Y$-periodic. Let $\tens{A}(y)\in L^\infty_\#(Y^*)^{N\times N}$ satisfy $\sum\limits_{i,j=1}^n\tens{A}_{ij}(y)\xi_i\xi_j\geq a\sum\limits_{i=1}^n\xi_i^2$ for all $\vec{\xi}\in\mathbf{R}^n$ for some $a>0$.\\
Consider the following boundary value problem for $\omega(\vec{y})$:
\begin{equation}
\begin{dcases}
-\nabla_{\vec{y}}\cdot\left(\tens{A}(\vec{y})\cdot\nabla_{\vec{y}}\omega\right)= F(\vec{y})&\text{ on }Y^*,\cr
-\left[\tens{A}(\vec{y})\nabla_{\vec{y}}\omega\right]\cdot\vec{n} = g(\vec{y})&\text{ on }\partial\mathcal{T},\cr
\omega\text{ is }Y\text{-periodic}.\cr
\end{dcases}\label{eq: simple}
\end{equation}
Then the following statements hold:
\begin{itemize}
    \item[(i)] There exists a weak $Y$-periodic solution $\omega\in H^1_\#(Y^*)/\mathbf{R}$ to (\ref{eq: simple}) if and only if $\int_{Y^*}F(\vec{y})\mathrm{d}\vec{y} = \int_{\partial \mathcal{T}}g(\vec{y})\mathrm{d}\sigma_y$.
    \item[(ii)] If (i) holds, then the uniqueness of weak solutions is ensured up to an additive constant.
\end{itemize}
\end{lemma}
See Lemma 2.1 in \cite{MunteanChalupecky2011}.\\
$\;$\\
Existence and uniqueness of the solutions of the systems (\ref{eq: L0sys}) and (\ref{eq: L1sys}) can be handled via the application of Rothe's method, see \cite{Rothe1984} for details on Rothe's method, and Gronwall's inequality, and see \cite{Dragomir2003} for various different versions of useful discrete Gronwall's inequalities. \\
$\;$\\
\begin{lemma}\label{l: v0-yindep}
The function $\vec{V}^0$ depends only on $(t,\vec{x})\in(0,T)\times\Omega$.
\end{lemma}
\begin{proof}
Applying Lemma \ref{l: simple} to system (\ref{eq: A0sys}) yields the weak solution $\vec{V}^0(t,x,y)\in H^1_\#(Y^*)/\mathbf{R}$ pointwise in $(t,\vec{x})\in(0,T)\times\Omega$ with uniqueness ensured up to an additive function depending only on $(t,\vec{x})\in(0,T)\times\Omega$. Direct testing of (\ref{eq: A0sys}) with $\vec{V}^0$ yields $\|\nabla_{\vec{y}}\vec{V}^0\|_{L^2_\#(Y^*)} = 0$. Hence, $\nabla_{\vec{y}}\vec{V}^0=0$ a.e. in $Y^*$.
\qed\end{proof}
\begin{corollary}\label{c: u0-yindep}
The function $\vec{U}^0$ depends only on $(t,\vec{x})\in(0,T)\times\Omega$.
\end{corollary}
\begin{proof}
Apply the gradient $\nabla_{\vec{y}}$ to system (\ref{eq: L0sys}). The independence of $\vec{y}$ follows directly from (A1) and Lemma \ref{l: v0-yindep}. \qed
\end{proof}
$\;$\\
The application of Lemma \ref{l: simple} to system (\ref{eq: A1sys}) yields, due to the divergence theorem, again a weak solution $\vec{V}^1(t,\vec{x},\vec{y})\in H^1_\#(Y^*)/\mathbf{R}$ pointwise in $(t,\vec{x})\in(0,T)\times\Omega$ with uniqueness ensured up to an additive function depending only on $(t,\vec{x})\in(0,T)\times\Omega$. One can determine $\vec{V}^1$ from $\vec{V}^0$ with the use of a decomposition of $V^1$ into products of $\vec{V}^0$ derivatives and so-called cell functions:
\begin{equation}
        \vec{V}^1 = \vec{W}\cdot\nabla_{\vec{x}}\vec{V}^0+\tens{Z}\vec{V}^0+\tilde{\vec{V}}^1\label{eq: cellV1}\\
\end{equation}
with $\nabla_{\vec{y}}\tilde{\vec{V}}^1 = \tens{0}$ and for $\alpha,\beta\in\{1,\ldots,N\}$ and $i\in\{1,\ldots,d\}$ with cell functions
\begin{equation}
    \tens{Z}_{\alpha\beta},W_{i}\in L^\infty(\mathbf{R}_+;W^{2,\infty}(\Omega;C^2_\#(Y^*)/\mathbf{R})).\label{eq: cellregularity}
\end{equation}
Insertion of (\ref{eq: cellV1}) into system (\ref{eq: A1sys}) leads to systems for the cell-functions $\vec{W}$ and $\tens{Z}$:
\begin{equation}\label{eq: cellWsys}
    \begin{dcases}
    \mathcal{A}^0\vec{W}=-\nabla_{\vec{y}}\cdot\tens{E}&\text{ in }Y^*,\cr
    \frac{\partial \vec{W}}{\partial\nu_{\tens{E}}} = -\vec{n}\cdot\tens{E}&\text{ on }\partial \mathcal{T},\cr
    \vec{W}\quad\text{ $Y$-periodic,}\cr
    \frac{1}{|Y|}\int_{Y^*}\vec{W}\mathrm{d}\vec{y} = \vec{0}.
    \end{dcases}
\end{equation}
and
\begin{equation}\label{eq: cellDsys}
    \begin{dcases}
    \mathcal{A}^0\tens{Z}=-\nabla_{\vec{y}}\cdot\tens{D}&\text{ in }Y^*,\cr
    \frac{\partial \tens{Z}}{\partial\nu_{\tens{E}}} = -\vec{n}\cdot\tens{D}&\text{ on }\partial \mathcal{T},\cr
    \tens{Z}_{\alpha\beta}\quad\text{ $Y$-periodic,}\cr
    \frac{1}{|Y|}\int_{Y^*}\tens{Z}\mathrm{d}\vec{y} = \tens{0}.
    \end{dcases}
\end{equation}
Again the existence and uniqueness up to an additive constant of the cell functions in systems (\ref{eq: cellWsys}) and (\ref{eq: cellDsys}) follow from Lemma \ref{l: simple} and convenient applications of the divergence theorem. The regularity of solutions follows from Theorem 9.25 and Theorem 9.26 in \cite{Brezis2010}.\\
$\;$\\
The existence and uniqueness for $\vec{V}^2$ follows from applying Lemma \ref{l: simple} to system (\ref{eq: A2sys}), which states that a solvability condition has to be satisfied. This solvability condition is the upscaled version of (\ref{eq: Aeps-sys}), the spatial partial differential equation for $\vec{V}^0$:
\begin{equation}    \label{eq: upscaledV}
\overline{\tens{M}}\vec{V}^0-\nabla_{\vec{x}}\cdot\left(\tens{E}^*\cdot\nabla_{\vec{x}}\vec{V}^0+\tens{D}^*\vec{V}^0\right) = \overline{\vec{H}}+\overline{\tens{K}}\vec{U}^0,
\end{equation}
where we have used (\ref{eq: cellV1}), the cell function decomposition, and the new short-hand notation
\begin{subequations}
\begin{align}
    \tens{E}^* &= \frac{1}{|Y|}\int_{Y^*}\tens{E}\cdot\left(\tens{1}+\nabla_{\vec{y}}\vec{W}\right)\mathrm{d}\vec{y},\label{eq: E*}\\
    \tens{D}^* &= \frac{1}{|Y|}\int_{Y^*}\tens{D}+\tens{E}\cdot\nabla_{\vec{y}}\tens{Z}\mathrm{d}\vec{y}.\label{eq: D*}
\end{align}
\end{subequations}
\begin{lemma}\label{l: u0v0-sys}
The pair $(\vec{U}^0,\vec{V}^0)\in H^1((0,T)\times\Omega)\times L^\infty((0,T);H^1_0(\Omega))$ are weak solutions to the following system
\begin{equation}    \label{eq: upscaledUV}
\begin{dcases}
\overline{\tens{M}}\vec{V}^0-\nabla_{\vec{x}}\cdot\left(\tens{E}^*\cdot\nabla_{\vec{x}}\vec{V}^0+\tens{D}^*\vec{V}^0\right) = \overline{\vec{H}}+\overline{\tens{K}}\vec{U}^0&\text{in }(0,T)\times\Omega,\cr
\frac{\partial \vec{U}^0}{\partial t}+\tens{L}\vec{U}^0=\tens{G}\vec{V}^0&\text{in }(0,T)\times\Omega,\cr
\vec{V}^0 = \vec{0}&\text{on }(0,T)\times\partial\Omega,\cr
\vec{U}^0 = \vec{U}^*&\text{on }\{0\}\times\Omega.
\end{dcases}
\end{equation}
\end{lemma}
\begin{proof}
From system (\ref{eq: A0sys}), equation (\ref{eq: upscaledV}), $\nabla_{\vec{y}}\vec{V}^0=\tens{0}$, assumption (A3) and system (\ref{eq: L0sys}), we see that $\nabla_{\vec{y}}\vec{U}^0=\tens{0}$. This leads automatically to system (\ref{eq: upscaledUV}), since there is no $\vec{y}$-dependence and $\Omega^\epsilon\subset\Omega$, $\Omega^\epsilon\rightarrow\Omega$, $\partial_{ext}\Omega^\epsilon=\partial\Omega$. Analogous to the proof of Theorem \ref{H: t: existuniq} we obtain the required spatial regularity. Moreover, by testing the second line with $\frac{\partial}{\partial t}\vec{U}^0$, applying a gradient to the second line and testing it with $\frac{\partial}{\partial t}\nabla\vec{U}^0$, we obtain the required temporal regularity as well.
\qed\end{proof}

\subsection{Combining two-scale convergence and asymptotic expansions}
\begin{theorem}\label{t: it-is-u0}
Let (A1)-(A3) be valid, then $(\vec{u},\vec{v}) = (\vec{U}^0,\vec{V}^0)$.
\end{theorem}
\begin{proof}
From \textbf{(P$^0_s$)} and Lemma \ref{l: u0v0-sys}, we see that $(\vec{u},\vec{v})$ and $(\vec{U}^0,\vec{V}^0)$ satisfy the same linear boundary value problem. We only have to prove the uniqueness for this boundary value problem.\\
From testing (\ref{eq: cellDsys}) with $\vec{W}$ and (\ref{eq: cellWsys}) with $\tens{Z}$, we obtain the identity
\begin{equation}
    \int_{Y^*}(\nabla_{\vec{y}}\vec{W})^\top\cdot\tens{D}\mathrm{d}\vec{y} = \int_{Y^*}\tens{E}\cdot\nabla_{\vec{y}}\tens{Z}\mathrm{d}\vec{y}.\label{eq: WD=Edel}
\end{equation}
Hence, from (\ref{eq: D*}) we get
\begin{equation}
        \tens{D}^*= \frac{1}{|Y|}\int_{Y^*}\left(\tens{1}+(\nabla_{\vec{y}}\vec{W})\right)^\top\cdot\tens{D}\mathrm{d}\vec{y}.\label{eq: D*new}
\end{equation}
Moreover, testing system (\ref{eq: cellWsys}) with $\vec{W}$ yields the identity
\begin{equation}
        \tens{E}^* = \frac{1}{|Y|}\int_{Y^*}\left(\tens{1}+(\nabla_{\vec{y}}\vec{W})\right)^\top\cdot\tens{E}\cdot\left(\tens{1}+(\nabla_{\vec{y}}\vec{W})\right)\mathrm{d}\vec{y}.\label{eq: E*new}
\end{equation}
We subtract \textbf{(P$^0_s$)} from (\ref{eq: upscaledUV}) and introduce $\tilde{\vec{U}}$, $\tilde{\vec{V}}$ as
\begin{equation}
    \tilde{\vec{U}} = \vec{U}^0-\vec{u}\quad\text{ and }\quad\tilde{\vec{V}} = \vec{V}^0-\vec{v}.\label{eq: utildedef}
\end{equation}
Testing with $\tilde{\vec{V}}$ and putting the $Y^*$-integral outside the $\Omega$-integral, we obtain the equation
\begin{equation}
    0 = \frac{1}{|Y|}\int_{Y^*}\left[\left\langle\tens{M}\tilde{\vec{V}},\tilde{\vec{V}}\right\rangle+\left\langle\tens{E}\cdot\vec{\zeta}+\tens{D}\tilde{\vec{V}},\vec{\zeta}\right\rangle - \left\langle\tens{K}\tilde{\vec{U}},\tilde{\vec{V}}\right\rangle\right]\mathrm{d}\vec{y}, \label{eq: zetasys}
\end{equation}
where
\begin{equation}
    \vec{\zeta} = \left(\tens{1}+(\nabla_{\vec{y}}\vec{W})\right)\cdot\nabla_{\vec{x}}\tilde{\vec{V}}.\label{eq: zetadef}
\end{equation}
This equation is identical to the Neumann problem (\ref{eq: Aeps-sys})-(\ref{eq: Vhole=0}) with $\vec{H}=\vec{0}$, $\tens{J}=\tens{0}$, and replacements $\nabla_{\vec{x}}\vec{V} \rightarrow \vec{\zeta}$, $\vec{U}\rightarrow\tilde{\vec{U}}$ and $\vec{V}\rightarrow\tilde{\vec{V}}$ in (\ref{eq: Aeps-sys}). Moreover, (\ref{eq: Aeps-sys}) is coercive due to assumption (A3). Therefore, we can follow the argument of the proof of Theorem \ref{H: t: existuniq}, but we only use equations (\ref{H: eq:  aprioriV}) and (\ref{eq: L1bound}) with constants $\tilde{H}$ and $\tilde{J}_{i\alpha}$ set to $0$. For some $R>0$, this leads to
\begin{equation}
    \frac{\partial}{\partial t}\|\tilde{\vec{U}}\|_{L^2(\Omega;L^2_\#(Y^*))^N}^2\leq R\|\tilde{\vec{U}}\|_{L^2(\Omega;L^2_\#(Y^*))^N}^2.\label{eq: tildeUineq}
\end{equation}
Applying Gronwall inequality and using the initial value $\tilde{\vec{U}} = \vec{U}^*-\vec{U}^*=\vec{0}$, we obtain $\|\tilde{\vec{U}}\|_{L^2(\Omega;L^2_\#(Y^*))^N}=0$ a.e. in $(0,T)$. By the coercivity, we obtain $\|\tilde{\vec{V}}\|_{L^2(\Omega;L^2_\#(Y^*))^N}=0$ and $\|\vec{\zeta}\|_{L^2(\Omega;L^2_\#(Y^*))^N}=0$.\\
From the proof of Proposition 6.12 in \cite{CioranescuDonato1999}, we see that $\tens{1}+\nabla_{\vec{y}}W$ does not have a kernel that contains non-zero $Y$-periodic solutions. Therefore, $\vec{\zeta}=\vec{0}$ yields $\nabla_{\vec{y}}\tilde{\vec{V}} = \tens{0}$. Thus, we have $\tilde{\vec{U}}=\vec{0}$ in $L^\infty((0,T);L^2(\Omega))^N$ and $\tilde{\vec{V}}=\vec{0}$ in $L^\infty((0,T);H^1_0(\Omega))^N$. Hence, $(\vec{u},\vec{v}) = (\vec{U}^0,\vec{V}^0)$.
\qed\end{proof}
\begin{corollary}\label{H: l: : aprioriHOM} Let $\lambda\geq0$ and $\tilde{\kappa}\geq0$ be as in Theorem \ref{H: t: existuniq}. Then there exists a positive constant $C$ independent of $\epsilon$ such that
\begin{equation}\label{eq: timebound}
    \|\vec{U}^0\|_{H^1(\Omega^\epsilon)^N}(t)\leq Ce^{\lambda t},\qquad \|\vec{V}^0\|_{\mathbb{V}_\epsilon^N}(t)\leq C(1+\tilde{\kappa}e^{\lambda t})
\end{equation}
    holds for $t\geq0$.
\end{corollary}
\begin{proof}
It is well known that bounded sequences converge weakly, and any weak limit adheres to the same bound. Since two-scale convergence implies weak convergence, the bounds of Theorem \ref{H: t: existuniq} hold for $\vec{U}^0$ and $\vec{V}^0$ as well.
\qed\end{proof}
This concludes the proof of Theorem \ref{t: two-scale}.
\section{Corrector estimates via asymptotic expansions} \label{s: corrector}
It is natural to determine the speed of convergence of the weak solutions $(\vec{U}^\epsilon,\vec{V}^\epsilon)$ to $(\vec{U}^0,\vec{V}^0)$. However, certain boundary effects are expected due to intersection of the external boundary with the perforated periodic cells. It is clear that $\Omega^\epsilon\rightarrow\Omega$ for $\epsilon\downarrow0$, but the boundary effects impact the periodic behavior, which can lead to $\vec{V}^j\neq\vec{0}$ at $\partial_{ext}\Omega^\epsilon$ for $j>0$. Hence, a cut-off function is introduced to remove this potentially problematic part of the domain.\\
Let us again introduce the cut-off function $M_\epsilon$ defined by
\begin{equation}
    \begin{dcases}
    M_\epsilon\in\mathcal{D}(\Omega),\cr
    M_\epsilon = 0 &\text{ if }\mathrm{dist}(\vec{x},\partial\Omega)\leq\epsilon,\cr
    M_\epsilon = 1 &\text{ if }\mathrm{dist}(\vec{x},\partial\Omega)\geq2\epsilon,\cr
    \epsilon\left|\frac{\mathrm{d} M_\epsilon}{\mathrm{d} x_i}\right|\leq C&i\in\{1,\ldots,d\}.
    \end{dcases}\label{eq: cutoff1}
\end{equation}
With this cut-off function defined, we introduce again the error functions
\begin{subequations}
\begin{align}
\vec{\Phi}^\epsilon&=\vec{V}^\epsilon-\vec{V}^0-M_\epsilon(\epsilon \vec{V}^1+\epsilon^2 \vec{V}^2),\label{eq: errorfunct1}\\
\vec{\Psi}^\epsilon&=\vec{U}^\epsilon-\vec{U}^0-M_\epsilon(\epsilon \vec{U}^1+\epsilon^2 \vec{U}^2),\label{eq: errorfunct2}
\end{align}
\end{subequations}
where the $M_\epsilon$ terms are the so-called corrector terms.
\subsection{Preliminaries}
\label{sec: prelim}
The solvability condition for system (\ref{eq: A2sys}) naturally leds to the fact that $(\vec{U}^0,\vec{V}^0)$ has to satisfy system (\ref{eq: upscaledUV}). Similar to solving system (\ref{eq: A1sys}) for $\vec{V}^1$, we handle system (\ref{eq: A2sys}) for $\vec{V}^2$ with a decomposition into cell-functions:
\begin{equation}
        \vec{V}^2 = \vec{P} + \tens{Q}^0\vec{V}^0+\tens{R}^0\vec{U}^0+\tens{Q}^1\cdot\nabla_{\vec{x}}\vec{V}^0+\tens{R}^1\cdot\nabla_{\vec{x}}\vec{U}^0+\tens{Q}^2:\mathrm{D}^2_{\vec{x}}\vec{V}^0\label{eq: cellV2}
\end{equation}
where we have the cell-functions
\begin{equation}
\begin{array}{rcl}
P_{\alpha},R_{\alpha\beta}^0,R_{i\alpha\beta}^1&\in& L^\infty(\mathbf{R}_+;W^{2,\infty}(\Omega^\epsilon;C^3_\#(Y^*))),\cr
Q_{\alpha\beta}^0,Q_{i\alpha\beta}^1&\in& L^\infty(\mathbf{R}_+;W^{2,\infty}(\Omega^\epsilon;C^2_\#(Y^*))),\cr
Q^2_{ij}&\in& L^\infty(\mathbf{R}_+;W^{2,\infty}(\Omega^\epsilon;C^2_\#(Y^*)))
\end{array}\label{eq: regularity3}
\end{equation}
for $\alpha,\beta\in\{1,\ldots,N\}$ and for $i,j\in\{1,\ldots,d\}$, and where
\begin{equation}
    (\tens{Q}^2:\mathrm{D}^2_{\vec{x}}\vec{V}^0)_\alpha := \sum_{i,j=1}^dQ_{ij}\frac{\partial^2 V^0_\alpha}{\partial x_i\partial x_j}.\label{eq: doubleD}
\end{equation}
The cell-functions $\vec{P}$, $\tens{Q}^0$, $\tens{R}^0$, $\tens{Q}^1$, $\tens{R}^1$, $\tens{Q}^2$ satisfy the following systems of partial differential equations, obtained from subtracting (\ref{eq: upscaledV}) from (\ref{eq: A2sys}) and inserting (\ref{eq: cellV2}):
\begin{equation}\label{eq: cellR0sys}
    \begin{dcases}
    \mathcal{A}^0\vec{P}=\vec{H}-\overline{\vec{H}}&\text{ in }Y^*,\cr
    \frac{\partial \vec{P}}{\partial\nu_{\tens{E}}} = \vec{0}&\text{ on }\partial \mathcal{T},\cr
    \vec{P}\quad\text{ $Y$-periodic,}
    \end{dcases}
\end{equation}
\begin{equation}\label{eq: cellR1sys}
    \begin{dcases}
    \mathcal{A}^0\tens{Q}^0=\nabla_{\vec{y}}\cdot\left(\tens{E}\cdot\nabla_{\vec{x}}\tens{Z}\right)+\nabla_{\vec{x}}\cdot\left(\tens{E}\cdot\nabla_{\vec{y}}\tens{Z}\right)+\nabla_{\vec{y}}\cdot\left(\tens{D}\tens{Z}\right)\cr
    \qquad\qquad+\nabla_{\vec{x}}\cdot\left(\tens{D}-\tens{D}^*\right)+\overline{\tens{M}}-\tens{M}&\text{ in }Y^*,\cr
    \frac{\partial \tens{Q}^0}{\partial\nu_{\tens{E}}} = -\left(\tens{D}\tens{Z}+\tens{E}\cdot\nabla_{\vec{x}}\tens{Z}\right)\cdot\vec{n}&\text{ on }\partial \mathcal{T},\cr
    \tens{Q}^0\quad\text{ $Y$-periodic,}
    \end{dcases}
\end{equation}
\begin{equation}\label{eq: cellR2sys}
    \begin{dcases}
    \mathcal{A}^0\tens{R}^0=\tens{K}-\overline{\tens{K}}&\text{ in }Y^*,\cr
    \frac{\partial R^0_{\alpha\beta}}{\partial\nu_{\tens{E}}} = \tens{0}&\text{ on }\partial \mathcal{T},\cr
    \tens{R}^0\quad\text{ $Y$-periodic,}
    \end{dcases}
\end{equation}
\begin{equation}\label{eq: cellR3sys}
    \begin{dcases}
    \mathcal{A}^0\tens{Q}^1=\nabla_{\vec{y}}\cdot(\tens{E}\cdot\nabla_{\vec{x}}\vec{W})\otimes\tens{1}+\nabla_{\vec{y}}\cdot(\tens{E}\otimes\tens{Z})\cr
    \qquad\qquad+\nabla_{\vec{x}}\cdot(\tens{E}\cdot\nabla_{\vec{y}}\vec{W})\otimes\tens{1}+\tens{E}\cdot\nabla_{\vec{y}}\tens{Z}\cr
    \qquad\qquad+\nabla_{\vec{y}}\cdot(\tens{D}\otimes\vec{W})+\nabla_{\vec{x}}\cdot(\tens{E}-\tens{E}^*)\otimes\tens{1}+\tens{D}-\tens{D}^*&\text{ in }Y^*,\cr
    \frac{\partial \tens{Q}^1}{\partial\nu_{\tens{E}}} = \vec{W}\otimes(\tens{D}\cdot\vec{n})+\vec{n}\cdot\left(\tens{E}\otimes\tens{Z}+\tens{E}\cdot\nabla_{\vec{x}}\vec{W}\otimes\tens{1}\right)&\text{ on }\partial \mathcal{T},\cr
    \tens{Q}^1\quad\text{ $Y$-periodic,}
    \end{dcases}
\end{equation}
\begin{equation}\label{eq: cellR4sys}
    \begin{dcases}
    \mathcal{A}^0\tens{R}^1=\tens{0}&\text{ in }Y^*,\cr
    \frac{\partial \tens{R}^1}{\partial\nu_{\tens{E}}} = \tens{0}&\text{ on }\partial \mathcal{T},\cr
    \tens{R}^1\quad\text{ $Y$-periodic,}
    \end{dcases}
\end{equation}
\begin{equation}\label{eq: cellthetasys}
    \begin{dcases}
    \mathcal{A}^0\tens{Q}^2\!=\!\nabla_{\vec{y}}\cdot(\tens{E}\otimes\vec{W})+\tens{E}\cdot\nabla_{\vec{y}}\vec{W}+\tens{E}-\tens{E}^*&\text{ in }Y^*,\cr
    \frac{\partial\tens{Q}^2}{\partial\nu_{\tens{E}}} = -\vec{n}\cdot\tens{E}\otimes \vec{W}&\text{ on }\partial \mathcal{T},\cr
    \tens{Q}^2\quad\text{ $Y$-periodic.}
    \end{dcases}
\end{equation}
The well-posedness of the cell-problems (\ref{eq: cellWsys}) - (\ref{eq: cellthetasys}) is given by Lemma \ref{l: simple}, while the regularity follows from Theorem 9.25 and Theorem 9.26 in \cite{Brezis2010}. Note that cell-problem (\ref{eq: cellR4sys}) yields $\tens{R}^1=\tens{0}$.
\subsection{Proof of Theorem \ref{t: corrector}}
\label{sec: proof1}
Let $C$ denote a constant independent of $\epsilon$, $\vec{x}$, $\vec{y}$ and $t$.\\
We rewrite the error-function $\vec{\Phi}^\epsilon$ as
\begin{equation}\label{eq: phiintro}
    \vec{\Phi}^\epsilon = \vec{V}^\epsilon-\vec{V}^0-M_\epsilon(\epsilon \vec{V}^1+\epsilon^2 \vec{V}^2) = \vec{\phi}^\epsilon+(1-M_\epsilon)(\epsilon \vec{V}^1+\epsilon^2\vec{V}^2),
\end{equation}
where
\begin{equation}\label{eq: phiexp}
    \vec{\phi}^\epsilon = \vec{V}^\epsilon-(\vec{V}^0+\epsilon \vec{V}^1+\epsilon^2\vec{V}^2).
\end{equation}
Similarly, we make use of the error-function $\vec{\Psi}^\epsilon$
\begin{equation}\label{eq: psiintro}
    \vec{\Psi}^\epsilon = \vec{U}^\epsilon-\vec{U}^0-M_\epsilon(\epsilon \vec{U}^1+\epsilon^2\vec{U}^2).
\end{equation}
The goal is to estimate both $\vec{\Phi}^\epsilon$ and $\vec{\Psi}^\epsilon$ uniformly in $\epsilon$.\\
$\;$\\
Even though our problem for $(\vec{U}^\epsilon,\vec{V}^\epsilon)$ is defined on $\Omega^\epsilon$, while the asymptotic expansion terms $(\vec{U}^i,\vec{V}^i)$ are defined on $\Omega\times Y^*$, we are still able to use spaces defined on $\Omega^\epsilon$ such as $\mathbb{V}_\epsilon^N$ since the evaluation $\vec{y}=\vec{x}/\epsilon$ transfers the zero-extension on $\mathcal{T}$ to $\mathcal{T}^\epsilon$.\\
$\;$\\
Introduce the coercive bilinear form $a_\epsilon: \mathbb{V}_\epsilon^N\times\mathbb{V}_\epsilon^N\rightarrow\mathbf{R}$ defined as \begin{equation}\label{eq: a-bilform}
    a_\epsilon(\vec{\psi},\vec{\phi}) = \int_{\Omega^\epsilon}\vec{\phi}^\top \mathcal{A}^\epsilon\vec{\psi}\mathrm{d}\vec{x}
\end{equation}
pointwise in $t\in\mathbf{R}_+$, on which it depends implicitly.\\
By construction, $\vec{\Phi}^\epsilon$ vanishes on $\partial_{ext}\Omega^\epsilon$, which allows for the estimation of $\|\vec{\Phi}^\epsilon\|_{\mathbb{V}_\epsilon^N}$. This estimation follows the standard approach, see \cite{CioranescuStJeanPaulin1998} for the details.\\
First the inequality $|a_\epsilon(\vec{\Phi}^\epsilon,\vec{\phi})|\leq \mathcal{C}(\epsilon,t)\|\vec{\phi}\|_{\mathbb{V}_\epsilon^N}$, where $\mathcal{C}(\epsilon,t)$ is a constant depending on $\epsilon$ and $t\in\mathbb{R}_+$, is obtained for any $\vec{\phi}\in \mathbb{V}_\epsilon^N$.
Second, we take $\vec{\phi} = \vec{\Phi}^\epsilon$ and using the coercivity, one immediately obtains $\|\vec{\Phi}^\epsilon\|_{\mathbb{V}_\epsilon^N}$.\\
Our pseudo-parabolic system complicates this approach. Instead of $\mathcal{C}(\epsilon,t)$, one gets $C\|\vec{\Psi}^\epsilon\|_{H^1_0(\Omega^\epsilon)^N}$. Via an ordinary differential equation for $\vec{\Psi}^\epsilon$, we obtain a temporal inequality for $\|\vec{\Psi}^\epsilon\|_{H^1_0(\Omega^\epsilon)^N}$ that contains $\|\vec{\Phi}^\epsilon\|_{\mathbb{V}_\epsilon^N}$. The upper bound for $\|\vec{\Phi}^\epsilon\|_{\mathbb{V}_\epsilon^N}$ now follows from applying Gronwall's inequality, leading to an upper bound for $\|\vec{\Psi}^\epsilon\|_{H^1_0(\Omega^\epsilon)^N}$.\\
$\;$\\
From equation (\ref{eq: phiintro}), we have
\begin{equation}\label{eq: aexpand}
a_\epsilon(\vec{\Phi}^\epsilon,\vec{\phi}) = a_\epsilon(\vec{\phi}^\epsilon,\vec{\phi})+a_\epsilon((1-M_\epsilon)(\epsilon \vec{V}^1+\epsilon^2\vec{V}^2),\vec{\phi})
\end{equation}
for $\vec{\phi}\in\mathbb{V}_\epsilon^N$.\\
Do note that $M_\epsilon$ vanishes in a neighbourhood of the boundary $\partial_{ext}\Omega^\epsilon$, see (\ref{eq: cutoff1}), because of which the second term in (\ref{eq: aexpand}) vanishes outside this neighbourhood.\\
We start by estimating the first term of (\ref{eq: aexpand}), $a_\epsilon(\vec{\phi}^\epsilon,\vec{\phi})$. From the asymptotic expansion of $\mathcal{A}^\epsilon$, we obtain
\begin{multline}
    \mathcal{A}^\epsilon\vec{\phi}^\epsilon = (\epsilon^{-2}\mathcal{A}^0+\epsilon^{-1}\mathcal{A}^1+\mathcal{A}^2)\vec{\phi}^\epsilon\\
    = \mathcal{A}^\epsilon \vec{V}^\epsilon - \epsilon^{-2}\mathcal{A}^0\vec{V}^0-\epsilon^{-1}(\mathcal{A}^0\vec{V}^1+\mathcal{A}^1\vec{V}^0)-(\mathcal{A}^0\vec{V}^2+\mathcal{A}^1\vec{V}^1+\mathcal{A}^2\vec{V}^0)\\
    -\epsilon(\mathcal{A}^1\vec{V}^2+\mathcal{A}^2\vec{V}^1)-\epsilon^2\mathcal{A}^2\vec{V}^2.\label{eq: Aepsexpansion1}
\end{multline}
Using the definitions of $\mathcal{A}^0$, $\mathcal{A}^1$, $\mathcal{A}^2$, $\vec{V}^0$, $\vec{V}^1$, $\vec{V}^2$, we have
\begin{equation}
    \mathcal{A}^\epsilon \vec{\phi}^\epsilon = \tens{K}^\epsilon \vec{U}^\epsilon-\left.\tens{K}\right|_{\vec{y}=\vec{x}/\epsilon}\vec{U}^0+\epsilon\tens{J}^\epsilon\nabla \vec{U}^\epsilon-\epsilon(\mathcal{A}^2\vec{V}^1+\mathcal{A}^1\vec{V}^2)-\epsilon^2\mathcal{A}^2\vec{V}^2.\label{eq: Aepsexpansion2}
\end{equation}
The function $\vec{\phi}^\epsilon$ satisfies the following boundary condition on $\partial\mathcal{T}^\epsilon$
\begin{equation}
    \frac{\partial\vec{\phi}^\epsilon}{\partial \nu_{\tens{D}^\epsilon}} = -\epsilon^2\frac{\partial \vec{V}^2}{\partial \nu_\tens{D}},\label{eq: phibound}
\end{equation}
as a consequence of the boundary conditions for the $\vec{V}^i$-terms. Hence, $\vec{\phi}^\epsilon$ satisfies the following system:
\begin{equation}
    \begin{dcases}
    \mathcal{A}^\epsilon \vec{\phi}^\epsilon = \vec{f}^\epsilon-\epsilon \vec{g}^\epsilon &\text{ in }\Omega^\epsilon,\cr
    \frac{\partial\vec{\phi}^\epsilon}{\partial \nu_{\tens{D}^\epsilon}}=\epsilon^2\tens{h}^\epsilon\cdot\vec{n}^\epsilon&\text{ on }\partial \mathcal{T}^\epsilon,\cr
    \vec{\phi}^\epsilon=-\epsilon\vec{V}^1-\epsilon^2\vec{V}^2&\text{ on }\partial\Omega.
    \end{dcases}\label{eq: phisystem}
\end{equation}
Testing with $\vec{\phi}^\top\in \mathbb{V}_\epsilon^N$ and performing a partial integration, we obtain
\begin{equation}\label{eq: phiweak}
    a_\epsilon(\vec{\phi}^\epsilon,\vec{\phi}) = \int_{\Omega^\epsilon} \vec{\phi}^\top\vec{f}^\epsilon\mathrm{d}\vec{x}-\int_{\Omega^\epsilon}\epsilon \vec{\phi}^\top\vec{g}^\epsilon\mathrm{d}\vec{x}+\int_{\partial \mathcal{T}^\epsilon}\epsilon^2\vec{\phi}^\top\tens{h}^\epsilon\cdot\vec{n}^\epsilon\mathrm{d}\vec{s},
\end{equation}
where $\vec{f}^\epsilon$, $\vec{g}^\epsilon$ and $\tens{h}^\epsilon$ are given by
\begin{equation}\label{eq: feps}
\vec{f}^\epsilon= \tens{K}^\epsilon \vec{U}^\epsilon-\left.\tens{K}\right|_{\vec{y}=\vec{x}/\epsilon}\vec{U}^0
\end{equation}
\begin{multline}\label{eq: geps}
    \vec{g}^\epsilon = \mathcal{A}^1\left[\vec{P}+\tens{Q}^0\vec{V}^0+\tens{R}^0\vec{U}^0+\tens{Q}^1\cdot\nabla_{\vec{x}}\vec{V}^0+\tens{R}^1\cdot\nabla_{\vec{x}}\vec{U}^0+\tens{Q}^2:\mathrm{D}_{\vec{x}}^2\vec{V}^0\right]\\ +\mathcal{A}^2\left[\vec{W}\cdot\nabla_{\vec{x}}\vec{V}^0+\tens{Z} \vec{V}^0\right]-\tens{J}^\epsilon\cdot\nabla_{\vec{x}}\vec{U}^\epsilon\\
    +\epsilon \mathcal{A}^2\left[\vec{P}+\tens{Q}^0\vec{V}^0+\tens{R}^0\vec{U}^0+\tens{Q}^1\cdot\nabla_{\vec{x}}\vec{V}^0+\tens{R}^1\cdot\nabla_{\vec{x}}\vec{U}^0+\tens{Q}^2:\mathrm{D}_{\vec{x}}^2\vec{V}^0\right],
\end{multline}
\begin{equation}\label{eq: heps}
    \tens{h}^\epsilon = -\frac{\partial}{\partial \nu_{\tens{D}}}\left[\vec{P}+\tens{Q}^0\vec{V}^0+\tens{R}^0\vec{U}^0+\tens{Q}^1\cdot\nabla_{\vec{x}}\vec{V}^0+\tens{R}^1\cdot\nabla_{\vec{x}}\vec{U}^0+\tens{Q}^2:\mathrm{D}_{\vec{x}}^2\vec{V}^0\right].
\end{equation}
Estimates for $\vec{f}^\epsilon$, $\vec{g}^\epsilon$ and $\tens{h}^\epsilon$ follow from estimates on $\vec{V}^0$, $\vec{U}^0$, $\vec{P}$, $\tens{Q}^0$, $\tens{R}^0$, $\tens{Q}^1$, $\tens{R}^1$, $\tens{Q}^2$, and $\vec{W}$. Due to the regularity of $\overline{\vec{H}}$, $\overline{\tens{K}}$, $\tens{J}$, $\tens{G}$, classical regularity results for elliptic systems, see Theorem 8.12 and Theorem 8.13 in \cite{GilbargTrudinger}, quarantee that all spatial derivatives up to the fourth order of $(\vec{U}^0,\vec{V}^0)$ are in $L^\infty(\mathbf{R}_+\times\Omega)$. Similarly, from Theorem 9.25 and Theorem 9.26 in \cite{Brezis2010}, the cell-functions $\vec{W}$, $\vec{P}$, $\tens{Q}^0$, $\tens{R}^0$, $\tens{Q}^1$, $\tens{R}^1$ and $\tens{Q}^2$ have higher regularity, than given by Lemma \ref{l: simple}: $W_i,P_\alpha,Q^0_{\alpha\beta},R^0_{\alpha\beta},Q^1_{i\alpha\beta},R^1_{\alpha\beta},Q^2_{ij}$ are in $L^\infty(\mathbf{R}_+;W^{2,\infty}(\Omega;H^3_\#(Y^*)/\mathbf{R}))$. We denote with $\kappa$ the time-independent bound \begin{equation*}
    \kappa=\sup\limits_{1\leq\alpha,\beta\leq N}\|K_{\alpha\beta}\|_{L^\infty(\mathbf{R}_+;W^{1,\infty}(\Omega;C^1_\#(Y^*)))}.
\end{equation*} Note that $\|\tens{R}^0\|_{L^\infty(\mathbf{R}_+\times\Omega;C^1_\#(Y^*)))^{N\times N}}\leq C\kappa$ by the Poincar\'{e}-Wirtinger inequality.\\
Bounding $\vec{g}^\epsilon$ follows now directly from equation (\ref{eq: geps}) and Corollary \ref{H: l: : aprioriHOM}:
\begin{equation}\label{eq: gn}
\|g^\epsilon_\alpha\|_{L^2(\Omega^\epsilon)^N}\leq C(1+\epsilon)(1+(\kappa+\tilde{\kappa})e^{\lambda t}),
\end{equation}
where $C$ is independent of $\epsilon.$\\
Bounding $\tens{h}^\epsilon$ is more difficult as it is defined on the boundary $\partial\mathcal{T}^\epsilon$. The following result, see Lemma 2.31 on page 47 in \cite{CioranescuStJeanPaulin1998}, gives a trace inequality, which shows that $\tens{h}^\epsilon$ is properly defined.
\begin{lemma}\label{l: lions}
Let $\psi\in H^1(\Omega^\epsilon)$. Then
\begin{equation}
    \|\psi\|_{L^2(\partial \mathcal{T}^\epsilon)}\leq C\epsilon^{-1/2}\|\psi\|_{\mathbb{V}_\epsilon},\label{eq: perforationtrace}
\end{equation}
where $C$ is independent of $\epsilon$.
\end{lemma}
By (\ref{eq: heps}), the regularity of the cell-functions, the regularity of the normal at the boundary, Corollary \ref{H: l: : aprioriHOM} and using Lemma \ref{l: lions} twice, we have
\begin{equation}\label{eq: hn}
    \left|\int_{\partial \mathcal{T}^\epsilon}\epsilon^2\vec{\phi}^\top\tens{h}^\epsilon\cdot\vec{n}^\epsilon\mathrm{d}s(\vec{x})\right|\leq C\epsilon(1+(\kappa+\tilde{\kappa})e^{\lambda t})\|\vec{\phi}\|_{\mathbb{V}_\epsilon^N}.
\end{equation}
We estimate $\vec{f}^\epsilon$ in $L^2(\Omega^\epsilon)^N$ from the standard inequality $|a_1b_1-a_2b_2|\leq|a_1-a_2||b_2|+|a_1||b_1-b_2|$ for all $a_1,a_2,b_1,b_2\in\mathbf{R}$. This leads to
\begin{multline}
    \|\vec{f}^\epsilon\|_{L^2(\Omega^\epsilon)^N}\leq \|\tens{K}^\epsilon-\tens{K}\|_{L^2(\Omega^\epsilon)^{N\times N}}\|\vec{U}^\epsilon\|_{L^\infty(\Omega^\epsilon)^N}\\
    +\|\tens{K}\|_{L^\infty(\Omega^\epsilon)^{N\times N}}\|\vec{U}^\epsilon-\vec{U}^0\|_{L^2(\Omega^\epsilon)^N}.\label{eq: Ksplit}
\end{multline}
With this inequality, the estimation depends on the convergence of $\tens{K}^\epsilon$ and $\vec{U}^\epsilon$ to $\tens{K}$ and $\vec{U}^0$, respectively,
but with the notation according to (\ref{eq: notation}) we have $\tens{K}^\epsilon-\left.\tens{K}\right|_{\vec{y}=\vec{x}/\epsilon}=\tens{0}$ a.e.\\
From the definition of $\vec{\Psi}^\epsilon$, we obtain
\begin{multline}
\|\vec{U}^\epsilon-\vec{U}^0\|_{L^2(\Omega^\epsilon)^N} = \|\vec{\Psi}^\epsilon+M_\epsilon(\epsilon \vec{U}^1+\epsilon^2 \vec{U}^2)\|_{L^2(\Omega^\epsilon)^N}\\
\leq \|\vec{\Psi}^\epsilon\|_{L^2(\Omega^\epsilon)^N}+\epsilon\|\vec{U}^1\|_{L^2(\Omega^\epsilon)^N}+\epsilon^2\|\vec{U}^2\|_{L^2(\Omega^\epsilon)^N}.\label{eq: Ucorrbound}
\end{multline}
Introduce the notations $l=L_N$ and $t_l=\min\{1/l,t\}$. Using system (\ref{eq: L1sys}), the bounds $C(1+(\kappa+\tilde{\kappa})e^{\lambda t})$ for $\|\vec{V}^1\|_{H^1(\Omega^\epsilon)^N}$ and $\|\vec{V}^2\|_{H^1(\Omega^\epsilon)^N}$ obtained via the cell-function decompositions (\ref{eq: cellV1}) and (\ref{eq: cellV2}), respectively, the inequalities (\ref{eq: L1bound}) and (\ref{eq: L2bound}), and by employing Gronwall's inequality, we obtain
\begin{subequations}
\begin{align}
\|\vec{U}^1\|_{H^1(\Omega^\epsilon)^N}&\leq C(1+(\kappa+\tilde{\kappa})e^{\lambda t})\sqrt{t_le^{lt}+t_l^2e^{2lt}},\label{eq: u1}\\
\|\vec{U}^2\|_{H^1(\Omega^\epsilon)^N}&\leq C(1+(\kappa+\tilde{\kappa})e^{\lambda t})\sqrt{t_le^{lt}+t_l^2e^{2lt}},\label{eq: u2}\\
\|\vec{U}^\epsilon\!-\!\vec{U}^0\|_{L^2(\Omega^\epsilon)^N}&\leq \|\vec{\Psi}^\epsilon\|_{L^2(\Omega^\epsilon)^N}\cr
&\quad+C(\epsilon+\epsilon^2)(1+(\kappa+\tilde{\kappa})e^{\lambda t})\sqrt{t_le^{lt}+t^2_le^{2lt}}.\label{eq: udiff}
\end{align}
\end{subequations}
Thus from identity (\ref{eq: Ksplit}) we obtain
\begin{multline}
    \|\vec{f}^\epsilon\|_{L^2(\Omega^\epsilon)^N}\leq \kappa\|\vec{\Psi}^\epsilon\|_{L^2(\Omega^\epsilon)^N}\\
    +C(\epsilon+\epsilon^2)\kappa(1+(\kappa+\tilde{\kappa})e^{\lambda t})\sqrt{t_le^{lt}+t_l^2e^{2lt}},\label{eq: fbound}
\end{multline}
We now have all the ingredients to estimate $a_\epsilon(\vec{\phi}^\epsilon,\vec{\phi})$. Inserting estimates (\ref{eq: gn}), (\ref{eq: hn}) and (\ref{eq: fbound}) into (\ref{eq: phiweak}), we find
\begin{multline}\label{eq: abound}
    |a_\epsilon(\vec{\phi}^\epsilon,\vec{\phi})|\!\leq\! \left[\kappa\|\vec{\Psi}^\epsilon\|_{L^2(\Omega^\epsilon)^N}\right.\\
    \left.\!C(\epsilon\!+\!\epsilon^2)(1\!+\!(\kappa+\tilde{\kappa})e^{\lambda t})(1+\kappa (1+ t_le^{lt}))\right]\!\|\vec{\phi}\|_{\mathbb{V}_\epsilon^N}.
\end{multline}
Next, we need to estimate the second right-hand term of (\ref{eq: aexpand}),  $a_\epsilon((1-M_\epsilon)(\epsilon \vec{V}^1+\epsilon^2 \vec{V}^2),\vec{\phi})$. Trusting \cite{CioranescuStJeanPaulin1998} (see pages 48 and 49 in the reference) and using the bounds $C(1+(\kappa+\tilde{\kappa})e^{\lambda t})$ for $\|\vec{V}^1\|_{H^1(\Omega^\epsilon)^N}$ and $\|\vec{V}^2\|_{H^1(\Omega^\epsilon)^N}$, we obtain
\begin{multline}
    |a_\epsilon((1-M_\epsilon)(\epsilon \vec{V}^1+\epsilon^2 \vec{V}^2),\vec{\phi})|\\
    \leq \left[C(\epsilon^{\frac{1}{2}}+\epsilon^{\frac{3}{2}})+C(\epsilon+\epsilon^2)\left(1+(\kappa+\tilde{\kappa})e^{\lambda t}\right)\right]\|\vec{\phi}\|_{\mathbb{V}_\epsilon^N}.\label{eq: mebound}
\end{multline}
The combination of (\ref{eq: abound}) and (\ref{eq: mebound}) yields
\begin{multline}
    |a_\epsilon(\vec{\Phi}^\epsilon,\vec{\phi})|\!\leq\! \left[\kappa\|\vec{\Psi}^\epsilon\|_{L^2(\Omega^\epsilon)^N}\!\right.\\
    \left.+C(\epsilon^{\frac{1}{2}}\!+\!\epsilon^{\frac{3}{2}})\left(1+\epsilon^{\frac{1}{2}}(1\!+\!(\kappa+\tilde{\kappa})e^{\lambda t})(1\!+\!\kappa (1+ t_le^{lt}))\right)\right]\!\|\vec{\phi}\|_{\mathbb{V}_\epsilon^N}.\label{eq: bilepsexp}
\end{multline}
Since $\mathcal{L}\vec{\Psi}^\epsilon = \tens{G}\vec{\Phi}^\epsilon$, we obtain an identity similar to (\ref{eq: L1bound}) to which we apply Gronwall's inequality, leading to
\begin{equation}
    \label{eq: Psi-bound}
    \|\vec{\Psi}^\epsilon\|_{L^2(\Omega^\epsilon)^N}^2(t)\leq \int_0^te^{l(t-s)}G_N\|\vec{\Phi}^\epsilon\|_{L^2(\Omega^\epsilon)^N}^2(s)\mathrm{d}s.
\end{equation}
Choosing $\vec{\phi}=\vec{\Phi}^\epsilon$ and with $m$ denoting the coercivity constant $\min\limits_{1\leq i\leq n,1\leq\alpha\leq N}\{\tilde{m}_\alpha,\tilde{e}_i\}$, we obtain
\begin{equation}\label{eq: Phi-prebound}
    m\|\vec{\Phi}^\epsilon\|_{\mathbb{V}_\epsilon^N}^2\!\leq\! \left[\kappa\sqrt{\int_0^te^{l(t-s)}G_N\|\vec{\Phi}^\epsilon\|_{L^2(\Omega^\epsilon)^N}^2(s)\mathrm{d}s}\!+\!\mathcal{B}(\epsilon,t)\right]\!\|\vec{\Phi}^\epsilon\|_{\mathbb{V}_\epsilon^N},
\end{equation}
where
\begin{equation}
    \mathcal{B}(\epsilon,t) = C(\epsilon^{\frac{1}{2}}\!+\!\epsilon^{\frac{3}{2}})\left(1+\epsilon^{\frac{1}{2}}(1\!+\!(\kappa+\tilde{\kappa})e^{\lambda t})(1\!+\!\kappa(1+ t_le^{lt}))\right).\label{eq: betadef}
\end{equation}
Applying Young's inequality twice, once with $\eta>0$ and once with $\eta_1>0$, using the Poincar\'{e} inequality (see Remark \ref{r: poincare}) and Gronwall's inequality to (\ref{eq: Phi-prebound}), we arrive at
\begin{multline}
    \|\vec{\Phi}^\epsilon\|_{\mathbb{V}_\epsilon^N}^2\!\leq\!\frac{\mathcal{B}(\epsilon,t)^2}{\eta_1(2m-\eta_1-\eta)}\\+\int_0^t\frac{\kappa^2 G_N e^{l(t-s)}\mathcal{B}(\epsilon,s)^2}{\eta(2m-\eta_1-\eta)^2\eta_1}\exp\left(\int_s^t\frac{\kappa G_N}{\eta(2m-\eta_1-\eta)}e^{l(t-u)}\mathrm{d}u\right)\mathrm{d}s.\label{eq: betabound}
\end{multline}
Since $0<\mathcal{B}(\epsilon,s)\leq \mathcal{B}(\epsilon,t)$ for $s\leq t$, we can use the Leibniz rule to obtain
\begin{equation}
    \|\vec{\Phi}^\epsilon\|_{\mathbb{V}_\epsilon^N}^2\!\leq\!\frac{\mathcal{B}(\epsilon,t)^2}{\eta_1(2m-\eta_1-\eta)}\exp\left(\frac{\kappa^2 G_N}{\eta(2m-\eta_1-\eta)}t_le^{lt}\right).\label{eq: phiboundbeta}
\end{equation}
Minimizing the two fractions separately leads us to $\eta_1 = m-\frac{\eta}{2}$ and $\eta = m-\frac{\eta_1}{2}$, whence $\eta=\eta_1 = \frac{2}{3}m$. Hence, we obtain
\begin{equation}\label{eq: finalboundPhi}
\begin{array}{rcl}
        \|\vec{\Phi}^\epsilon\|_{\mathbb{V}_\epsilon^N}\!&\leq&\! C(\epsilon^{\frac{1}{2}}\!\!+\!\epsilon^{\frac{3}{2}}\!)\left(1\!+\!\epsilon^{\frac{1}{2}}\!(1\!+\!(\kappa\!+\!\tilde{\kappa})e^{\lambda t})(1\!+\!\kappa (1\!+\! t_le^{lt}))\right)\!\exp\!\left(\mu t_le^{lt}\right)\!,\\
        &=&\mathcal{C}(\epsilon,t)
        \end{array}
\end{equation}
and from (\ref{eq: Psi-bound}), we arrive at
\begin{equation}
        \|\vec{\Psi}^\epsilon\|_{H^1(\Omega^\epsilon)^N}\leq\!\mathcal{C}(\epsilon,t)\sqrt{t_le^{lt}}\label{eq: finalboundPsi}
\end{equation}
with
\begin{equation}\label{eq: mudef1}
    \mu = \frac{9\kappa^2G_N}{8m^2}.
\end{equation}
This completes the proof of Theorem \ref{t: corrector}.
\qed
\section{Upscaling and convergence speeds for a concrete corrosion model}\label{s: application}
In \cite{Vromans2018PAC} a concrete corrosion model has been derived from first principles. This model combines mixture theory with balance laws, while incorporating chemical reaction effects, mechanical deformations, incompressible flow, diffusion, and moving boundary effects. The model represents the onset of concrete corrosion by representing the corroded part as a layer of cement (the mixture) on top of a concrete bed and below an acidic fluid. The mixture contains three components $\phi = (\phi_1,\phi_2,\phi_3)$, which react chemically via $3+2\rightarrow 1$. For simplification, we work in volume fractions. Hence, the identity $\phi_1+\phi_2+\phi_3 = 1$ holds. The model equations on a domain $\Omega$ become for $\alpha\in\{1,2,3\}$
\begin{subequations}
\begin{align}
\frac{\partial\phi_\alpha}{\partial t}+\epsilon\nabla\cdot(\phi_\alpha\mathbf{v}_\alpha)-\epsilon\delta_\alpha\Delta \phi_\alpha&=\quad\epsilon\kappa_\alpha \mathcal{F}(\phi_1,\phi_3),\label{eq: cor1}\\
\nabla\cdot\left(\sum_{\alpha=1}^3\phi_\alpha\mathbf{v}_\alpha\right)-\sum_{\alpha=1}^3\delta_\alpha\Delta \phi_\alpha&=\quad\sum_{\alpha=1}^3\kappa_\alpha \mathcal{F}(\phi_1,\phi_3),\label{eq: cor2}\\
\nabla(-\phi_\alpha p+\![\lambda_\alpha\!+\!\mu_\alpha]\nabla\!\!\cdot\!\mathbf{u}_\alpha)\!+\!\mu_\alpha\Delta \mathbf{u}_\alpha&=\chi_\alpha(\mathbf{v}_\alpha\!-\!\mathbf{v}_3)\!-\!\!\sum_{\beta=1}^3\!\gamma_{\alpha\beta}\Delta \mathbf{v}_\beta,\label{eq: cor3}\\
\nabla\left(-p+\sum_{\alpha=1}^2(\lambda_\alpha+\mu_\alpha)\nabla\cdot\mathbf{u}_\alpha\right)+\sum_{\alpha=1}^2\mu_\alpha\Delta \mathbf{u}_\alpha&+\sum_{\alpha=1}^3\sum_{\beta=1}^3\gamma_{\alpha\beta}\Delta \mathbf{v}_\beta = 0,\label{eq: cor4}
\end{align} \end{subequations}
where $U_\alpha$ and $\vec{v}_\alpha=\partial U_\alpha/\partial t$ are the displacement and velocity of component $\alpha$, respectively, and $\epsilon$ is a small positive number independent of any spatial scale. Equation (\ref{eq: cor1}) denotes a mass balance law, (\ref{eq: cor2}) denotes the incompressibility condition, (\ref{eq: cor3}) the partial (for component $\alpha$) momentum balance law, and (\ref{eq: cor4}) the total momentum balance.\\
For $t=\mathcal{O}(\epsilon^0)$, we can treat $\phi$ as constant, which removes some nonlinearities from the model. Moreover, with equation (\ref{eq: cor2}) we can eliminate $\vec{v}_3$ in favor of $\vec{v}_1$ and $\vec{v}_2$, while with equation (\ref{eq: cor4}) we can eliminate $p$. This leads to a final expression for $\vec{u} = (U_1,U_2)$:
\begin{equation}
      \tilde{\tens{M}}\partial_t\vec{u} -\tilde{\tens{A}}\vec{u}-\mathrm{div}\left(\tilde{\tens{B}}\vec{u}+\tilde{\tens{D}}\partial_t\vec{u}+
      \tens{E}\cdot\nabla\left(\tens{F}\vec{u}+\tilde{\tens{G}}\partial_t\vec{u}\right)\right)= \vec{H},\label{eq: pseudopara}
\end{equation}
with
\begin{subequations}
\begin{align}
\tilde{\tens{M}} &= \left(\begin{array}{cc}
     \chi_1\frac{\phi_1+\phi_3}{\phi_3}&\chi_1\frac{\phi_2}{\phi_3}\\
     \chi_2\frac{\phi_1}{\phi_3}&\chi_2\frac{\phi_2+\phi_3}{\phi_3}
\end{array}\right),&\quad\tilde{\tens{A}}&=\tilde{\tens{B}} = \tilde{\tens{D}} = \tens{0},\label{eq: constants1}\\
\tens{F} &= \left(\begin{array}{cc}
     \mu_1(\phi_2+\phi_3)&-\mu_2\phi_1\\
     -\mu_1\phi_2&\mu_2(\phi_1+\phi_3)
\end{array}\right),&\quad\tens{E} &= \mathbf{I},\label{eq: constants2}\\
\tilde{G}_{\alpha\beta} &= -\gamma_{\alpha\beta}+\phi_\alpha\sum_{\lambda=1}^3\gamma_{\lambda\beta},&\quad H_\alpha &= \frac{\chi_\alpha}{\phi_3}\mathcal{F}(\phi_1,\phi_3)\sum_{\lambda = 1}^3\kappa_\lambda.\label{eq: constants3}
\end{align}
\end{subequations}
According to \cite{Vromans2018PAC}, there are several options for $\gamma_{\alpha\beta}$, but all these options lead to non-invertible $\mathcal{G}$. Suppose we take $\gamma_{11} = \gamma_{22}=\gamma_1<0$ and $\gamma_{12} = \gamma_{21}=\gamma_2<0$ with $\gamma_1>\gamma_2$. Then $\mathcal{G}$ is invertible and positive definite for $\phi_3>0$, since the determinant of $\mathcal{G}$ equals $(\gamma_1^2-\gamma_2^2)\phi_3$.\\
According to Section 4.3 of \cite{VromansLIC}, we obtain the Neumann problem (\ref{eq: Aeps-sys}), (\ref{eq: L-sys}) with
\begin{equation}
    \tens{M} = \tilde{\tens{M}}\tilde{\tens{G}}^{-1},\; \tens{D} = \tens{0},\;\tens{L} = \tilde{\tens{G}}^{-1}\tens{F},\; \tens{G} = \tilde{\tens{G}}^{-1},\; \tens{K} = -\tilde{\tens{M}}\tilde{\tens{G}}^{-1}\tens{F},\; \tens{J} = \tens{0}. \label{eq: appconst}
\end{equation}
Note, that both $\tens{E}$ and $\vec{H}$ do not change in this transformation. Moreover, $\tens{M}$ is positive definite, since both $\tilde{\tens{M}}$ and $\tilde{\tens{G}}$ are positive definite.\\
$\;$\\
Suppose the cement mixture has a periodic microstructure, satisfying assumption (A4), inherited from the concrete microstructure if corroded. Assume the constants $\chi_\alpha$, $\mu_\alpha$, $\kappa_\alpha$, and $\gamma_{\alpha\beta}$ are actually functions of both the macroscopic scale $\vec{x}$ and the microscopic scale $\vec{y}$, such that Assumptions (A1)-(A3) are satisfied. Note that (A3) is trivially satisfied.\\
$\;$\\
From the main results we see that a macroscale limit $(\vec{U}^0,\vec{V}^0)$ of this microscale corrosion problem exists, which satisfies system (\textbf{P}$^0_w$), and that the convergence speed is given by Theorem \ref{t: corrector} with constants $l$, $\kappa$, $\lambda$ and $\mu$ given by Appendix \ref{app: constants}.
\appendix
\section{Determining $\kappa$, $\tilde{\kappa}$ and exponents $l$, $\lambda$ and $\mu$.}\label{app: constants}
In Theorem \ref{t: corrector}, the three constants $l$, $\lambda$ and $\mu$ are introduced as exponents indicating the exponential growth in time of the corrector bounds. Moreover, there was also a constant $\kappa$ that indicated whether additional exponential growth occurs or not. For brevity it was not stated how these constants depend on the given matrices and tensors. Here we will give an exact determination procedure of these constants.\\
The constant $\kappa$ denotes the maximal operator norm of the tensor $\tens{K}$.
\begin{equation}
\kappa=\sup\limits_{1\leq\alpha,\beta\leq N}\|K_{\alpha\beta}\|_{L^\infty(\mathbf{R}_+;W^{1,\infty}(\Omega;C^1_\#(Y^*)))}.\label{eq: app kappabound}
\end{equation}
The constants $l$, $\lambda$, $\tilde{\kappa}$ and $\mu$ were obtained via Young's inequality, which make them a coupled system via several additional positive constants: $\eta$, $\eta_1$, $\eta_2$, $\eta_3$. The obtained expressions are
\begin{subequations}
\begin{align}
    l&= \max\{0,L_N\},\label{eq: app lbound}\\
    \lambda &= \frac{1}{2}\max\left\{0,L_N+\max\left\{L_G+G_M\max_{1\leq\alpha\leq N}\tilde{K}_\alpha,G_M\max_{1\leq\alpha\leq N}\max_{1\leq i\leq d}\tilde{J}_{i\alpha}\right\}\right\},\label{eq: app lambdabound}\\
    \mu&= \frac{9\kappa^2}{8m^2}G_N,\label{eq: app mubound}\\
    \tilde{\kappa} &= \max_{1\leq\alpha\leq N,1\leq i\leq d}\{\tilde{K}_\alpha,\tilde{J}_{i\alpha}\}\label{eq: app tilkapbound}
\end{align}
\end{subequations}
with the values
\begin{subequations}
\begin{align}
    L_N&= 2\mathcal{L}_{\min}+\eta G_{\max}+\eta_1dN\mathcal{L}_G,\label{eq: app LNbound}\\
    L_G&= 2\mathcal{L}_{\min}+\frac{dN}{\eta_1}\mathcal{L}_G+\eta_2G_{\max}+\eta_3dNG_G,\label{eq: app LGbound}\\
    G_N&=\frac{1}{\eta}G_{\max}+\frac{dN}{\eta_3}G_G,\label{eq: app GNbound}\\
    G_G&=\frac{1}{\eta_2}G_{\max} ,\label{eq: app GGbound}\\
    G_M&= \max_{1\leq\alpha\leq N}\max_{1\leq i\leq d}\left\{\frac{G_N+G_G}{\tilde{m}_\alpha},\frac{G_N}{\tilde{e}_i}\right\},\label{eq: app GMbound}\\
    m&= \min_{1\leq\alpha\leq N}\min_{1\leq i\leq d}\{\tilde{m}_\alpha,\tilde{e}_i\},\label{eq: app mbound}
\end{align}
\end{subequations}
where we have the positive values
\begin{subequations}
\begin{align}
    \tilde{m}_\alpha&=m_\alpha-\sum_{i=1}^d\sum_{\beta=1}^N\frac{\|D_{i\beta\alpha}\|_{L^\infty(\mathbf{R}_+\times\Omega;C_\#(Y^*))}}{2\eta_{i\beta\alpha}}-\eta_\alpha-\sum_{\beta=1}^N\eta_{\alpha\beta}-\sum_{i=1}^d\sum_{\beta=1}^N\tilde{\eta}_{i\alpha\beta},\label{eq: app tildembound}\\
    \tilde{e}_i&=e_i-\sum_{\alpha,\beta=1}^N\frac{\eta_{i\beta\alpha}}{2}\|D_{i\beta\alpha}\|_{L^\infty(\mathbf{R}_+\times\Omega;C_\#(Y^*))},\label{eq: app tildeebound}\\
    \tilde{H}&= \sum_{\alpha = 1}^N\frac{1}{4\eta_{\alpha}}\|\tens{H}_{\alpha}\|^2_{L^\infty(\mathbf{R}_+\times\Omega;C_\#(Y^*))},\label{eq: app tildehhbound}\\
    \tilde{K}_\alpha&= \sum_{\beta=1}^N\frac{1}{4\eta_{\beta\alpha}}\|K_{\beta\alpha}\|^2_{L^\infty(\mathbf{R}_+\times\Omega;C_\#(Y^*))},\label{eq: app tildehbound}\\
    \tilde{J}_{i\alpha}&= \sum_{\beta=1}^N\frac{\epsilon_{0}^2}{4\tilde{\eta}_{i\beta\alpha}}\|J_{i\beta\alpha}\|^2_{L^\infty(\mathbf{R}_+\times\Omega;C_\#(Y^*))}\label{eq: app tildehibound}
\end{align}
\end{subequations}
for $\eta_{i\beta\alpha}>0$, $\eta_{\beta}>0$, $\eta_{\alpha\beta}>0$, $\tilde{\eta}_{i\alpha\beta}>0$ and $\epsilon_{0}$ the supremum of allowed $\epsilon$ values (which is 1 for Theorem \ref{t: timeinterval}).
Moreover, we have\begin{itemize}
    \item $\mathcal{L}_{\min}$ as the $L^\infty(\mathbf{R}_+\times\Omega)$-norm of the absolute value of the largest negative eigenvalue or it is -1 times the smallest positive eigenvalue of $\tens{L}$ if no negative or 0 eigenvalues exist,
    \item $\mathcal{L}_G$ as the $L^\infty(\mathbf{R}_+\times\Omega)$-norm of the largest absolute value of the $\nabla\tens{L}$ components,
    \item $G_{\max}$ as the $L^\infty(\mathbf{R}_+\times\Omega)$-norm of the largest eigenvalue of $\tens{G}$,
    \item $G_G$ as the $L^\infty(\mathbf{R}_+\times\Omega)$-norm of the largest absolute value of the $\nabla\tens{G}$ components.
\end{itemize}
\begin{remarkArt}\label{r: ratecoupling}
Remark that smaller $l$ and $\mu$ yield longer times $\tau$ in Theorem \ref{t: timeinterval} and faster convergence rates in $\epsilon$. However, $l$ and $\mu$ are only coupled via $\lambda$. Hence, $l$ and $\mu$ can be made as small as needed as long as $\lambda$ remains finite and independent of $\epsilon$.
\end{remarkArt}
\begin{remarkArt}\label{r: minTree}
Note that $\mathcal{L}_{\min}<0$ allows for a hyperplane of positive values of $\eta$ and $\eta_1$ in $(\eta,\eta_1,\eta_2,\eta_3)$-space such that $l=L_N=0$. In this case not $\lambda$ or $\mu$ should be minimized. Instead $\tau_{end}$ should be maximized, the time $\tau$ for which the bounds of Theorem \ref{t: timeinterval} equal $\mathcal{O}(1)$ for $p=q=0$. For $\mu\geq\lambda$ this yields a minimization of $\mu$, while for $\mu<\lambda$ a minimization of $\mu+\lambda$. Due to the use of maximums in the definition of $\lambda$ and $\tau_{end}$, we refrain from maximizing $\tau_{end}$ as any attempt leads to a large tree of cases for which an optimization problem has to be solved.
\end{remarkArt}
\section{Two-scale convergence}\label{a: 2-scale}
Two-scale convergence is a method invented in 1989 by Nguetseng, see \cite{Nguetseng1989}. This method removes many technicalities by basing the convergence itself on functional analytic grounds as a property of functions in certain spaces. In some sense the function spaces natural to periodic boundary conditions have nice convergence properties of their oscillating continuous functions. This is made precise in the First Oscillation Lemma:
\begin{lemma}[`First Oscillation Lemma']\label{H: t: oscillation}
Let $B_p(\Omega,Y)$, $1\leq p<\infty$, denote any of the spaces $L^p(\Omega;C_\#(Y))$, $L^p_\#(\Omega;C(\overline{Y}))$, $C(\overline{\Omega};C_\#(Y))$. Then $B_p(\Omega,Y)$ has the following properties:
\begin{enumerate}
    \item $B_p(\Omega,Y)$ is a separable Banach space.
    \item $B_p(\Omega,Y)$ is dense in $L^p(\Omega\times Y)$.
    \item If $f(\vec{x},\vec{y})\in B_p(\Omega,Y)$. Then $f(\vec{x},\vec{x}/\epsilon)$ is a measurable function on $\Omega$ such that
    \begin{equation}
        \left\|f\left(\vec{x},\frac{\vec{x}}{\epsilon}\right)\right\|_{L^p(\Omega)}\leq  \left\|f\left(\vec{x},\vec{y}\right)\right\|_{B_p(\Omega,Y)}.\label{eq: appFOL1}
    \end{equation}
    \item For every $f(\vec{x},\vec{y})\in B_p(\Omega,Y)$, one has
    \begin{equation}
        \lim_{\epsilon\rightarrow0}\int_\Omega f\left(\vec{x},\frac{\vec{x}}{\epsilon}\right)\mathrm{d}\vec{x} = \frac{1}{|Y|}\int_\Omega\int_Yf(\vec{x},\vec{y})\mathrm{d}\vec{y}\mathrm{d}\vec{x}.\label{eq: appFOL2}
    \end{equation}
    \item For every $f(\vec{x},\vec{y})\in B_p(\Omega,Y)$, one has
    \begin{equation}
        \lim_{\epsilon\rightarrow0}\int_\Omega \left|f\left(\vec{x},\frac{\vec{x}}{\epsilon}\right)\right|^p\mathrm{d}\vec{x} = \frac{1}{|Y|}\int_\Omega\int_Y |f(\vec{x},\vec{y})|^p\mathrm{d}\vec{y}\mathrm{d}\vec{x}.\label{eq: appFOL3}
    \end{equation}
\end{enumerate}
See Theorems 2 and 4 in \cite{LukkassenNguetsengWall2002}.
\end{lemma}
However, application of the First Oscillation Lemma is not sufficient as it cannot be applied to weak solutions nor to gradients. Essentially two-scale convergence overcomes these problems by extending the First Oscillation Lemma in a weak sense.
\subsection{Two-scale convergence: definition and results}\label{s: two-scale}
For each function $c(t,\vec{x},\vec{y})$ on $(0,T)\times\Omega\times Y$, we introduce a corresponding sequence of functions $c^\epsilon(t,\vec{x})$ on $(0,T)\times\Omega$ by
\begin{equation}
c^\epsilon(t,\vec{x}) = c\left(t,\vec{x},\frac{\vec{x}}{\epsilon}\right)    \label{H: eq:  correspondence}
\end{equation}
for all $\epsilon\in(0,\epsilon_0)$, although two-scale convergence is valid for more general bounded sequences of functions $c^\epsilon(t,\vec{x})$.\\
Introduce the notation $\nabla_{\vec{y}}$ for the gradient in the $\vec{y}$-variable. Moreover, we introduce the notations $\rightarrow$, $\rightharpoonup$, and $\overset{2}{\longrightarrow}$ to point out strong convergence, weak convergence, and two-scale convergence, respectively.\\
\\
The two-scale convergence was first introduced in \cite{Nguetseng1989} and popularized with the seminal paper \cite{Allaire1992}, in which the term two-scale convergence was actually coined. For our explanation we use both the seminal paper \cite{Allaire1992} as the modern exposition of two-scale convergence in \cite{LukkassenNguetsengWall2002}. From now on, $p$ and $q$ are real numbers such that $1<p<\infty$ and $1/p+1/q=1$.
\begin{definition}\label{d: two-scale} Let $(\epsilon_h)_h$ be a fixed sequence of positive real numbers\footnote{when it is clear from the context we will omit the subscript $h$} converging to 0. A sequence $(u_\epsilon)$ of functions in $L^p(\Omega)$ is said to two-scale converge to a limit $u_0\in L^p(\Omega\times Y)$ if
\begin{equation}
    \int_\Omega u_\epsilon(\vec{x})\phi\left(\vec{x},\frac{\vec{x}}{\epsilon}\right)\mathrm{d}\vec{x}\rightarrow\frac{1}{|Y|}\int_\Omega\int_Yu_0(\vec{x},\vec{y})\phi(\vec{x},\vec{y})\mathrm{d}\vec{y}\mathrm{d}\vec{x},\label{H: eq:  two-scale}
\end{equation}
for every $\phi\in L^q(\Omega;C_\#(Y))$.\\
See Definition 6 on page 41 of \cite{LukkassenNguetsengWall2002}.
\end{definition}
We now list several important results concerning the two-scale convergence.
\begin{proposition}\label{H: p: grad-extra}
Let $(u_\epsilon)$ be a bounded sequence in $W^{1,p}(\Omega)$
for $1<p\leq\infty$ such that
 \begin{equation}
     u_\epsilon\rightharpoonup u_0\quad{in}\quad W^{1,p}(\Omega).\label{eq: apptwoscale1}
 \end{equation}
 Then $u_\epsilon\overset{2}{\longrightarrow}u_0$ and there exist a subsequence $\epsilon'$ and a $u_1\in L^p(\Omega;W^{1,p}_\#(Y)/\mathbf{R})$ such that
 \begin{equation}
     \nabla u_{\epsilon'}\overset{2}{\longrightarrow}\nabla u_0+\nabla_{\vec{y}}u_1.\label{eq: apptwoscale2}
 \end{equation}
 \end{proposition}
Proposition \ref{H: p: grad-extra} for $1<p<\infty$ is Theorem 20 in \cite{LukkassenNguetsengWall2002}, while for $p=2$ it is identity (i) in Proposition 1.14 in \cite{Allaire1992}. On page 1492 of \cite{Allaire1992} it is mentioned that the $p=\infty$ case holds as well. The case of interest for us here is $p=2$.

\begin{proposition}\label{H: p: grad}
Let $(u_\epsilon)$ and $(\epsilon\nabla u_\epsilon)$ be two bounded sequence in $L^2(\Omega)$. Then there exists a function $u_0(\vec{x},\vec{y})$ in $L^2(\Omega;H^1_\#(Y))$ such that, up to a subsequence, $u_\epsilon\overset{2}{\longrightarrow} u_0(\vec{x},\vec{y})$ and $\epsilon\nabla u_\epsilon\overset{2}{\longrightarrow} \nabla_{\vec{y}}u_0(\vec{x},\vec{y})$. See identity (ii) in Proposition 1.14 in \cite{Allaire1992}.
\end{proposition}

\begin{corollary}\label{H: c: 2-scale}
Let $(u_\epsilon)$ be a bounded sequence in $L^p(\Omega)$, with $1<p\leq\infty$. There exists a function $u_0(\vec{x},\vec{y})$ in $L^p(\Omega\times Y)$ such that, up to a subsequence, $u_\epsilon\overset{2}{\longrightarrow} u_0(\vec{x},\vec{y})$, i.e., for any function $\psi(\vec{x},\vec{y})\in\tens{D}(\Omega;C^\infty_\#(Y))$, we have
\begin{equation}
    \lim_{\epsilon\rightarrow0}\int_\Omega u_\epsilon(\vec{x})\psi\!\left(\vec{x},\frac{\vec{x}}{\epsilon}\right)\mathrm{d}\vec{x} = \frac{1}{|Y|}\int_\Omega\int_Yu_0(\vec{x},\vec{y})\psi(\vec{x},\vec{y})\mathrm{d}\vec{y}\mathrm{d}\vec{x}.\label{eq: apptwoscale3}
\end{equation}
See Corollary 1.15 in \cite{Allaire1992}.
\end{corollary}
\begin{theorem}\label{H: t: tripprod}
Let $(u_\epsilon)$ be a sequence in $L^p(\Omega)$ for $1<p<\infty$, which two-scale converges to $u_0\in L^p(\Omega\times Y)$ and assume that
\begin{equation}\label{H: eq:  limit}
    \lim_{\epsilon\rightarrow0}\|u_\epsilon\|_{L^p(\Omega)} = \|u_0\|_{L^p(\Omega\times Y)}.
\end{equation}
Then, for any sequence $(v_\epsilon)$ in $L^q(\Omega)$ with $\frac{1}{p}+\frac{1}{q} =1$, which two-scale converges to $v_0\in L^q(\Omega\times Y)$, we have that
\begin{equation}
    \int_\Omega u_\epsilon(\vec{x})v_\epsilon(\vec{x})\tau\!\left(\vec{x},\frac{\vec{x}}{\epsilon}\right)\mathrm{d}\vec{x}\rightarrow\int_\Omega\frac{1}{|Y|}\int_Y u_0(\vec{x},\vec{y})v_0(\vec{x},\vec{y})\tau(x,\vec{y})\mathrm{d}\vec{y}\mathrm{d}\vec{x},\label{eq: apptwoscale4}
\end{equation}
for every $\tau$ in $\tens{D}(\Omega,C^\infty_\#(Y))$. Moreover, if the $Y$-periodic extension of $u$ belong to $L^p(\Omega;C_\#(Y))$, then
\begin{equation}
    \lim_{\epsilon\rightarrow0}\left\|u_\epsilon(\vec{x})-u_0\left(\vec{x},\frac{\vec{x}}{\epsilon}\right)\right\|_{L^p(\Omega)}=0.\label{eq: apptwoscalestrong}
\end{equation}
See Theorem 18 in \cite{LukkassenNguetsengWall2002}.
\end{theorem}
These results generalize properties 3, 4 and 5 of the First Oscillation Lemma in such a way that the convergence applies to weak solutions, products and gradients AND it even guarantees that the convergence is strong for oscillating continuous functions.\\
Hence, two-scale convergence is suitable for upscaling problems.



\bibliographystyle{spmpsci}      
\bibliography{references1}   

%
%

\end{document}